\documentclass [12pt]{amsart}
\usepackage[utf8]{inputenc}
\pdfoutput=1
\usepackage{amsmath,amssymb,amsthm,amsfonts}
\usepackage{mathtools}%

\usepackage[english]{babel}%
\usepackage{comment}%
\usepackage{hyperref}%
\usepackage{bm}%
\usepackage{bbm}%
\usepackage{mathrsfs}%

\usepackage{tikz,graphicx,color}
\usepackage{tikz-cd}%
\usepackage{tikz-3dplot}%
\usetikzlibrary{calc}%
\usetikzlibrary{arrows}%
\usetikzlibrary{shapes}%
\usetikzlibrary{patterns}%
\usetikzlibrary{positioning}%
\usetikzlibrary{arrows.meta}
\usetikzlibrary{decorations.markings}
\usetikzlibrary{knots}

\usepackage{epstopdf}%

\usepackage[arrow]{xy}%
\usepackage{diagbox}%
\usepackage{subfig}%
\usepackage{arcs}%
\usepackage{xcolor}%
\usepackage{xspace}

\usepackage[margin=1.0in]{geometry}%

\usepackage{enumitem}
\usepackage{letltxmacro}
\usepackage{thmtools,etoolbox}

\def\myarabic#1{\normalfont(\roman{#1})}
\newlist{theoremlist}{enumerate}{1}
\setlist[theoremlist]{label=\myarabic{theoremlisti},ref={\myarabic{theoremlisti}},itemindent=0pt,labelindent=0pt,
  leftmargin=*,noitemsep}

\makeatletter
\renewcommand{\p@theoremlisti}{\perh@ps{\thetheorem}}
\protected\def\perh@ps#1#2{\textup{#1#2}}
\newcommand{\itemrefperh@ps}[2]{\textup{#2}}
\newcommand{\itemref}[1]{\begingroup\let\perh@ps\itemrefperh@ps\ref{#1}\endgroup}
\makeatother

\usepackage{nameref,hyperref}
\usepackage[capitalize]{cleveref}

\usepackage{bbm}%
\usepackage{bm}

\newtheorem{theorem}{Theorem}[section]

\newtheorem{lemma}[theorem]{Lemma}
\newtheorem{proposition}[theorem]{Proposition}
\newtheorem{corollary}[theorem]{Corollary}
\theoremstyle{definition}
\newtheorem{notation}[theorem]{Notation}
\theoremstyle{definition}
\newtheorem{remark}[theorem]{Remark}
\theoremstyle{definition}
\newtheorem{definition}[theorem]{Definition}
\newtheorem{conjecture}[theorem]{Conjecture}

\theoremstyle{definition}

\theoremstyle{definition}
\newtheorem{example}[theorem]{Example}

\crefname{figure}{Figure}{Figures}

\def\figref#1(#2){Figure~\hyperref[#1]{\ref*{#1}(#2)}}

\addtotheorempostheadhook[theorem]{\crefalias{theoremlisti}{theorem}}
\addtotheorempostheadhook[lemma]{\crefalias{theoremlisti}{lemma}}
\addtotheorempostheadhook[proposition]{\crefalias{theoremlisti}{proposition}}
\addtotheorempostheadhook[corollary]{\crefalias{theoremlisti}{corollary}}

\def\Acal{\mathcal{A}}\def\Bcal{\mathcal{B}}\def\Dcal{\mathcal{D}}\def\Fcal{\mathcal{F}}\def\Ical{\mathcal{I}}\def\Mcal{\mathcal{M}}\def\Xcal{\mathcal{X}}

\def\pbf{\mathbf{p}}

\def\C{{\mathbb{C}}}
\def\R{{\mathbb{R}}}

\def\Z{{\mathbb{Z}}}

\newcommand\parr[1]{{({#1})}}
\def\<{{\langle}}
\def\>{{\rangle}}

\def\la{{\lambda}}

\def\RP{{\R P}}

\def\Vert{{ \operatorname{Vert}}}

\def\proj{ \operatorname{proj}}

\def\wt{\operatorname{wt}}

\def\RR{{\mathbb R}}

\def\RP{{\RR\mathbb P}}

\def\xrasim{\xrightarrow{\sim}}

\def\Gr{\operatorname{Gr}}
\def\Grtnn{\Gr_{\ge 0}}

\def\Povtp_#1{\Pi_{#1}^{>0}}
\def\Povtnn_#1{\Pi_{#1}^{\geq0}}
\def\BND{\Bcal}
\def\Bound{\BND}%
\def\Boundkn{\BND(k,n)}

\newcounter{todocnt} %
\newcounter{todoex} %

\newcounter{todofigure}

\numberwithin{equation}{section}

\def\Space{\Xcal}
\def\Closure{\overline{\Space}}

\def\bth{{\bm{\theta}}}
\def\btht{{\tilde{\bm{\theta}}}}

\def\bv{{\bm{v}}}

\def\bx{{\mathbf{x}}}

\crefname{figure}{Figure}{Figures}

\def\th{\theta}
\def\tht{{\tilde\theta}}

\def\ggp_#1{\gamma^{\C}_{f',#1}}

\def\match{\Acal}

\def\Ptp_#1{\Pi^{>0}_{#1}}
\def\Ptnn_#1{\Pi^{\geq0}_{#1}}

\def\fb{{\bar f}}
\def\t{u}

\def\Measop{\operatorname{Meas}}

\def\Meas(#1,#2){\Measop_{#1}(#2)}
\def\Measp(#1,#2){\Measop'_{#1}(#2)}

\def\Bkn{\Boundkn}

\def\m#1{b_{#1}^{-}}
\def\p#1{b_{#1}^{+}}

\def\Crit{\operatorname{Crit}}
\def\Ctp{\Crit^{>0}}

\makeatletter
\newcommand{\raisemath}[1]{\mathpalette{\raisem@th{#1}}}
\newcommand{\raisem@th}[3]{\raisebox{#1}{$#2#3$}}
\makeatother

\def\Ctnn{\Crit^{\geq0}}
\def\Ctnnkn{\Crit^{\geq0}_{k,n}}

\def\br[#1]{[\![#1]\!]}
\def\brx[#1]{|#1|}

\def\Measd(#1,#2){\mathrm{M\widehat{e\vphantom{i}a}s}_{#1}(#2)}

\def\xrasim{\xrightarrow{\sim}}

\def\t{t}

\def\CCrit{\Crit}
\def\Cio{\CCrit^{\circ}}

\def\v{v}

\def\CioR_#1{\Cio_{#1}(\R)}

\def\fkn{{f_{k,n}}}

\def\conn_#1{c_{#1}}

\def\D_#1{D_{#1}}

\def\THtp{\Theta^{>0}}

\def\Hyp_#1{\Delta_{#1}}

\def\Cl(#1){#1^\boxtimes}

\let\ge\geqslant
\let\geq\geqslant

\let\leq\leqslant

\def\Meascl(#1,#2){{\overline{\Measop}}_{#1}(#2)}

\def\dsh#1{#1^\downarrow}
\def\Rtp{\R_{>0}}
\def\Rtnn{\R_{\geq0}}

\def\Pmid_#1{\Pi^{>0}_{#1,\dsh{#1}}}

\def\Measf{\Measop_f}

\def\paragraph#1{\subsubsection*{#1}}

\def\It{\tilde I}
\def\Icalt{\tilde \Ical}

\def\Gall{\mathcal{G}}
\def\Gred{\Gall_{\operatorname{red}}}

\def\Gkn{G_{k,n}}
\def\crat(#1,#2;#3,#4){(\v_{#1},\v_{#2};\v_{#3},\v_{#4})}

\makeatletter
\@namedef{subjclassname@2020}{\textup{2020} Mathematics Subject Classification}
\makeatother

\def\GH_#1{G^{\Mcal}_{#1}}

\def\Ordop{{\mathscr{O}}}
\def\COrdop{{\mathscr{L}}}
\def\Ord{\Ordop}
\def\COrd{\COrdop}

\def\Orda{\Ordop}

\def\Ass{{\mathscr{A}}}
\def\Cyc{{\mathscr{C}}}

\def\Oao{\Ord^\circ}

\def\tube{\tau}
\def\Tubing{\mathbf{T}}
\def\FAss(#1){\mathcal{F}_{\Ass}(#1)}
\def\KAss(#1){\mathcal{K}_{\Ass}(#1)}
\def\KCyc(#1){\mathcal{K}_{\Cyc}(#1)}
\def\FCyc(#1){\mathcal{F}_{\Cyc}(#1)}

\def\Melt{\Mcal}

\def\FM_#1{F_{\Melt}(#1)}

\def\Triples(#1){#1^{\<3\>}}

\def\HAT{\widehat\Tubing}

\def\al_#1{\alpha_{#1}}

\def\PTTx_#1{P_{#1}}
\def\Oo{\Ordop^\circ}

\def\D{\Dcal}

\def\Comp{\operatorname{Comp}}

\def\CHx(#1,#2){\widehat#2[#1]}
\def\CH(#1){\CHx(#1,\Tubing)}
\def\Px[#1]{#1}

\def\Afford(#1){\R^{|#1|-2}}

\def\Closure(#1){\overline{#1}}

\def\coord(#1,#2){z_{#2}(#1)}
\def\coordx(#1,#2){x_{#2}(#1)}
\def\coordxp(#1,#2){x'_{#2}(#1)}
\def\coordxt(#1,#2){\tilde x_{#2}(#1)}
\def\cxi_#1(#2,#3){x^\parr{#1}_{#3}(#2)}
\def\cyi_#1(#2,#3){y^\parr{#1}_{#3}(#2)}
\def\czi_#1(#2,#3){z^\parr{#1}_{#3}(#2)}
\def\coordy(#1,#2){y_{#2}(#1)}
\def\coordyt(#1,#2){\tilde y_{#2}(#1)}
\def\coordypn(#1,#2){y^\parr n_{#2}(#1)}
\def\coordxpn(#1,#2){x^\parr n_{#2}(#1)}
\def\coordz(#1,#2){z_{#2}(#1)}
\def\coordzp(#1,#2){z'_{#2}(#1)}

\def\ATTx(#1){A(#1)}
\def\BTTx(#1){B(#1)}

\def\OFace(#1,#2){\Fcal^\circ_{\Ord}(#1;#2)}
\def\COFace(#1,#2){\Fcal^\circ_{\COrd}(#1;#2)}

\def\T{\Tubing}
\def\t{\tube}

\def\avg{\operatorname{avg}}

\def\tp{\tau_+}

\def\tubes{{\mathbf{B}}}

\def\Tubes{\mathbf{Tubes}}

\def\tm{\t_-}

\def\EQ[#1]{\overline{#1}}
\def\eq[#1]{\overline{#1}}

\def\lp{\prec_P}

\def\leqp{\preceq_P}
\def\lPa{\prec_{\Pa}}
\def\Paf{\Pa_f}
\def\Pafb{\Pa_{\fbi}}
\def\lpaf{\prec_{\Paf}}
\def\lpa{\prec_{\Pa}}

\def\leqPa{\preceq_{\Pa}}
\def\leqpa{\preceq_{\Pa}}

\def\EQTubes(#1){[\Tubes(#1)]}

\def\Px[#1]{#1}
\def\Pax[#1]{#1}
\def\Pafx[#1]{#1}

\def\EQTubes(#1){[\Tubes(#1)]}

\def\Pa{\tilde P}
\def\Paf{\Pa_f}
\def\Pakn{\Pa_{k,n}}

\def\Cyco{\Cyc^\circ}

\def\sinmap{\zeta^\circ}
\def\sinmapb{\zeta}

\def\Neighv{\pch_G(v)}
\def\Neighb{\pch_G(b)}
\def\NeighXb_#1{\pch_{#1}(b)}

\def\Measbop{\overline{\operatorname{Meas}}}
\def\Measbf{\Measbop_f}
\def\Measbkn{\Measbop_{k,n}}
\def\Measkn{\Measop_{k,n}}
\def\RPtnn{\RP_{\ge0}}
\def\RPtp{\RP_{>0}}
\def\lpa{\lPa}

\def\pari{^\parr m}
\def\lpaf{\prec_{\Paf}}

\def\leqpakn{\preceq_{\Pakn}}
\def\mapCD{\phi}

\def\pbf{{\bm p}}
\def\pch{\pbf}
\def\pchn{{\pbf_{(n)}}}
\def\divsum{\xi}

\def\BB{{\mathbf{B}}}
\def\BBmod{{\bar{\mathbf{B}}}}
\def\B{B}
\def\Bmod{{\bar{B}}}

\def\mapDC{\psi}
\def\Gkn{G_{k,n}}
\def\gbth{g_\bth}
\def\gBB{g_\BB}
\def\gb{\bar g}
\def\gbbth{\gb_\bth}
\def\gbBB{\gb_\BB}

\def\ov{\operatorname{ov}}
\def\ovl{\ov_L}
\def\ovr{\ov_R}

\def\A{A}

\def\type(#1){type~{\normalfont(#1)}}
\def\types(#1,#2){types~{\normalfont(#1)} and~{\normalfont(#2)}}
\def\E(#1){\operatorname{E}(#1)}
\def\S(#1){\operatorname{S}(#1)}
\def\NW(#1){\operatorname{NW}(#1)}

\def\VB{V_{\bullet}(\Gkn)}

\def\AL{\ovr(\A,\B)}
\def\COL{\ovl(\A,\B)}

\def\rt{t}
\def\dist{{\operatorname{d}}_\bth}

\def\Rbth{R_{\bth}}
\def\BMX(#1){\Bmod^{(#1)}}

\def\RSZ{\R_{\Sigma=0}}

\def\bx{{\bm x}}
\def\bv{{\bm v}}

\def\by{{\bm y}}

\def\embop{\rho}
\def\res_#1{\embop_{#1}}

\def\ResPa{\tilde\embop}

\def\tbr[#1]{[#1]}

\def\prodb{\bar\prod}

\def\proj{\pi}
\def\psz{\proj_{\Sigma=0}}
\def\avg{\operatorname{avg}}

\def\D_#1{\Comp_{#1}(\Pa)}
\def\DT{\D_{\Tubing}}

\def\Df_#1{\Comp_{#1}(\Paf)}
\def\DfT{\Df_{\Tubing}}

\newcommand{\precdot}{\prec\mathrel{\mkern-3.5mu}\mathrel{\cdot}}

\def\pat{^\parr t}

\def\fbi{f}

\begin{document}
\numberwithin{equation}{section}

\title{Totally nonnegative critical varieties}
\author{Pavel Galashin}
\address{Department of Mathematics, University of California, Los Angeles, CA 90095, USA}
\email{{\href{mailto:galashin@math.ucla.edu}{galashin@math.ucla.edu}}}
\thanks{P.G.\ was supported by an Alfred P. Sloan Research Fellowship and by the National Science Foundation under Grants No.~DMS-1954121 and No.~DMS-2046915.}
\date{\today}

\subjclass[2020]{
  Primary:
  14M15. %
  Secondary:
  52B99, %
  15B48, %
  82B27, %
  05E99. %
}

\keywords{Critical varieties, totally nonnegative Grassmannian, cyclohedron, affine poset, hypersimplex, compactification}

\begin{abstract}
We study totally nonnegative parts of critical varieties in the Grassmannian. We show that each totally nonnegative critical variety $\Ctnn_f$ is the image of an affine poset cyclohedron under a continuous map and use this map to define a boundary stratification of $\Ctnn_f$. For the case of the top-dimensional positroid cell, we show that the totally nonnegative critical variety $\Ctnn_{k,n}$ is homeomorphic to the second hypersimplex $\Delta_{2,n}$.
\end{abstract}

\maketitle

\section*{Introduction}
The \emph{totally nonnegative Grassmannian} $\Grtnn(k,n)$ is a certain subset of the real Grassmannian introduced in~\cite{Pos,Lus2,LusIntro}. Recent years have revealed a variety of surprising connections between the structure of $\Grtnn(k,n)$ and statistical mechanics~\cite{CoWi, Lam, GP}, physics of scattering amplitudes~\cite{AHT,abcgpt},  and soliton solutions to the KP equation~\cite{kodama_williams_14}. In a recent paper~\cite{crit}, we introduced \emph{critical varieties} inside the Grassmannian, which may be considered ``critical parts'' of  \emph{positroid varieties} introduced in~\cite{KLS}. The construction of critical varieties is based on Kenyon's critical dimer model~\cite{Kenyon} and simultaneously includes the embeddings of the critical Ising model and critical electrical networks into $\Grtnn(k,n)$ discovered in~\cite{Lam,GP}.

Our aim in~\cite{crit} was to develop a theory of critical varieties which would parallel the theory of positroid varieties. For example, we introduced complex-algebraic \emph{open critical varieties} $\Cio_f$ as well as their totally positive parts $\Ctp_f$ called \emph{critical cells}. The goal of the present paper is to continue this program and study \emph{totally nonnegative critical varieties} $\Ctnn_f$, defined as closures of critical cells $\Ctp_f$ inside $\Grtnn(k,n)$.

While investigating the structure of the spaces $\Ctnn_f$, we were led to consider several new families of polytopes generalizing \emph{order polytopes}~\cite{Stanley_two}, \emph{associahedra}~\cite{Tamari,Stasheff,Haiman,Lee}, and \emph{cyclohedra}~\cite{BoTa,Simion}. We introduced \emph{poset associahedra} and \emph{affine poset cyclohedra} and explored their properties in~\cite{crit_polyt}. An important result from the point of view of applications to critical varieties is that these polytopes arise as compactifications of certain configuration spaces of points on a line and on a circle, analogously to the cases of associahedra and cyclohedra~\cite{AxSi,Sinha,LTV}. 

The goal of this paper is to prove two results on totally nonnegative critical varieties $\Ctnn_f$. First, we show that each space $\Ctnn_f$ is the image of an affine poset cyclohedron under a surjective continuous map. This observation, which may be considered an analog of the results of~\cite{PSW}, allows us to introduce a boundary stratification of $\Ctnn_f$. (Unlike in the case of positroid cells, the boundary stratification of $\Ctnn_f$ is not merely obtained by intersecting $\Ctnn_f$ with various positroid cells; see \cref{ex:intro:1}.) Next, we concentrate on the special case of the totally nonnegative critical variety $\Ctnn_{k,n}$ corresponding to the \emph{top-dimensional positroid cell} inside $\Grtnn(k,n)$. We show that $\Ctnn_{k,n}$ is homeomorphic to a polytope, namely, to the \emph{second hypersimplex $\Delta_{2,n}$}, via a stratification-preserving homeomorphism. 

As a surprising consequence, we see that $\Ctnn_{k,n}$ does not depend on $k$ as a stratified space. We view this result as a step towards constructing a family of conjectural \emph{shift maps} $\Gr(k,n)\dashrightarrow \Gr(k+1,n)$, which should restrict to homeomorphisms $\Ctnn_{k,n}\xrasim \Ctnn_{k+1,n}$. Constructing such shift maps is of great importance in relation to physics and statistical mechanics. For example, it would yield a connection between electrical networks and the Ising model (see~\cite[Question~9.2]{GP}) as well as provide insight into the construction of the \emph{BCFW triangulation}~\cite{BCFW} of the \emph{amplituhedron}~\cite{AHT}; see~\cite{abcgpt,GL,LPW,PSBW} and~\cite[Section~8]{crit} for context and related results.

Recall that the totally nonnegative parts of positroid varieties, while not being isomorphic to polytopes as stratified spaces, have remarkably simple topological structure~\cite{Wil,PSW, RW08,RW10,GKL,GKL3}. It remains an open problem to determine whether each totally nonnegative critical variety $\Ctnn_f$ is isomorphic to a polytope as a stratified space.

\section{Main results}\label{sec:intro:main}
We give a brief overview of some of our results. The full statements and proofs are given in the main body of the paper.

 Let $G$ be a planar graph embedded in a disk. We assume that $G$ has $n$ black degree $1$ boundary vertices labeled $b_1,b_2,\dots,b_n$ in clockwise order; see \figref{fig:plabic}(a). A \emph{strand} in $G$ is a path that makes a sharp right (resp., left) turn at each black (resp., white) vertex; see \figref{fig:plabic}(b). For each $p\in[n]:=\{1,2,\dots,n\}$, if a strand starts at the boundary vertex $b_p$, it must terminate at some boundary vertex $b_{\fbi_G(p)}$. The resulting permutation $\fbi_G\in S_n$ is called the \emph{strand permutation} of $G$. We say that $G$ is \emph{reduced}~\cite{Pos} if it has the minimal number of faces among all graphs with strand permutation $\fbi_G$.

For $0\leq k\leq n$, the \emph{totally nonnegative Grassmannian $\Grtnn(k,n)$} is the subset of the real Grassmannian $\Gr(k,n)$ where all Pl\"ucker coordinates have the same sign; see \cref{sec:backr:plabic} for further background. To a weight function $\wt:E(G)\to\Rtp$ defined on the edges of $G$, Postnikov~\cite{Pos} associates a point $\Meas(G,\wt)\in\Grtnn(k,n)$, where $0\leq k\leq n$ depends only on $G$.

In order to define a \emph{critical cell} $\Ctp_G$, we restrict to a special family of weight functions coming from the \emph{critical dimer model} of~\cite{Kenyon}. We will always assume that $G$ is reduced, in which case the critical cell $\Ctp_G$ depends only on the strand permutation of $G$ and is denoted $\Ctp_{\fbi_G}$.

For the purposes of this introduction, we consider the most important special case of the \emph{top cell} strand permutation $\fbi_{k,n}$. By definition, $\fbi_{k,n}\in S_n$ sends $p\mapsto p+k$ modulo $n$, for all $p\in[n]$. Let $\THtp_{k,n}$ be the space of $n$-tuples $\bv:=(v_1,v_2,\dots,v_n)\in\C^n$ of distinct points ordered counterclockwise on the unit circle, considered modulo global rotations of the circle. 
\begin{remark}\label{rmk:exp_simplex}
The space $\THtp_{k,n}$ is naturally homeomorphic to the interior of an $(n-1)$-dimensional simplex
\begin{equation*}%
  \THtp_{k,n}\cong\{\bth=(\th_1,\th_2,\dots,\th_n)\in\R^n\mid 0=\th_1<\th_2<\cdots<\th_n<\pi\},
\end{equation*}
by setting $v_r:=\exp(2i\th_r)$ for all $r\in[n]$. (In particular, $\THtp_{k,n}$ does not depend on $k$.)
\end{remark}

\begin{figure}
  \includegraphics[width=0.8\textwidth]{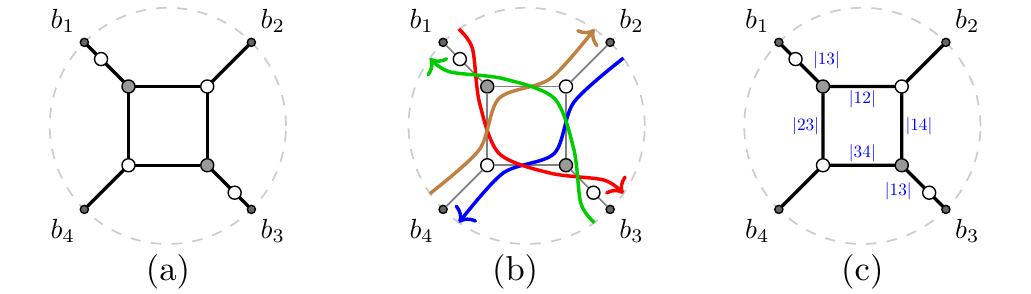}
  \caption{\label{fig:plabic} (a) A (reduced) planar bipartite graph $G$; (b) strands in $G$; (c)~edge weights $\wt_\bv(e)$, where the unmarked edges have weight $1$ and we abbreviate $\brx[pq]:=|v_p-v_q|$. Figure reproduced from~\cite[Figure~1]{crit}.}
\end{figure}

Every edge $e$ of $G$ belongs to exactly two strands. Denoting the endpoints of these strands by $b_p,b_q$ for $p,q\in[n]$, we say that $e$ is \emph{labeled by $\{p,q\}$}. In this case, we define its weight by
\begin{equation}\label{eq:intro:wt_bv}
  \wt_\bv(e):=
  \begin{cases}
    |v_p-v_q|, &\text{if $e$ is not incident to a boundary vertex;}\\
    1, &\text{otherwise.}
  \end{cases}
\end{equation}
We obtain a weight function $\wt_\bv:E(G)\to\Rtp$. See \figref{fig:plabic}(c) for an example. It turns out that the resulting point $\Meas(G,\wt_\bv)\in\Grtnn(k,n)$ does not depend on the choice of $G$. We denote $\Measkn(\bv):=\Meas(G,\wt_\bv)$. The \emph{critical cell} $\Ctp_{k,n}$ is defined as
\begin{equation*}%
  \Ctp_{k,n}:=\{\Measkn(\bv)\mid \bv\in\THtp_{k,n}\}.
\end{equation*}
Throughout, we assume that $2\leq k\leq n-1$. (For $k=1$ or $k=n$, $\Ctp_{k,n}$ is a single point.) According to~\cite[Theorem~1.10]{crit}, the map $\Measkn$ restricts to a homeomorphism $\THtp_{k,n}\xrasim \Ctp_{k,n}$, and thus $\Ctp_{k,n}$ is homeomorphic to the interior of an $(n-1)$-simplex. Our goal is to study the \emph{closure} $\Ctnn_{k,n}$ of $\Ctp_{k,n}$ inside $\Grtnn(k,n)$, and more generally, the closure $\Ctnn_\fbi$ of an arbitrary critical cell $\Ctp_\fbi$, $\fbi\in S_n$.

Informally, since $\Ctp_{k,n}$ is parameterized by configurations of $n$ distinct points on a circle, its closure $\Ctnn_{k,n}$ should be parameterized by $n$-point configurations where some points are allowed to collide. The map $\Measop_G$ is invariant under \emph{gauge transformations}: given a weighted graph $(G,\wt)$, for each interior vertex $u$ of $G$, rescaling the weights of all edges incident to $u$ by the same nonzero scalar does not alter the image of $\wt$ under $\Measop_G$. Modulo gauge transformations, $\Measkn(\bv)$ depends only on the ratios of the distances between the points $v_1,v_2,\dots,v_n$. For instance, even if all points $v_1,v_2,\dots,v_n$ collide together, it could happen that after we apply gauge transformations at the vertices of $G$, in the limit \emph{none} of the edge weights tend to zero, as the following example demonstrates.

\begin{figure}
  \includegraphics[width=0.8\textwidth]{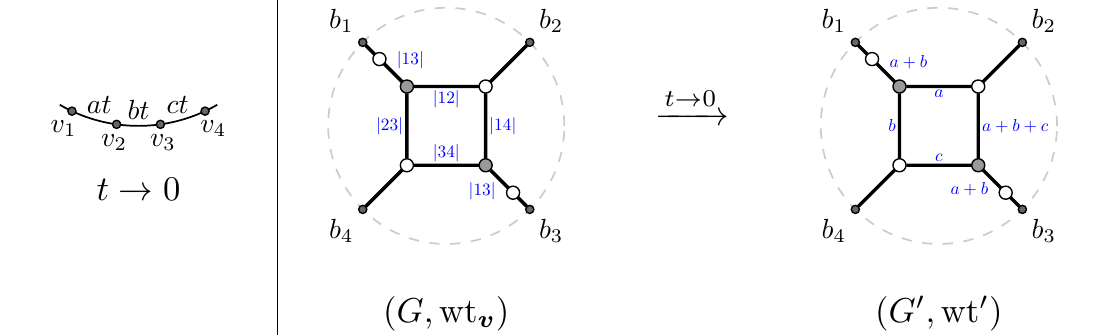} 
  \caption{\label{fig:limit1} Taking a limit where all points in $\bv$ collide. See \cref{ex:intro:1}.}
\end{figure}

\begin{example}\label{ex:intro:1}
Consider the graph $G$ in \cref{fig:plabic}, and suppose that $v_1,v_2,v_3,v_4$ collide in such a way that 
\begin{equation*}%
  (|v_2-v_1|:|v_3-v_2|:|v_3-v_1|:|v_4-v_3|:|v_4-v_2|:|v_4-v_1|)\to(a:b:a+b:c:b+c:a+b+c),
\end{equation*}
for some constants $a,b,c>0$; see \figref{fig:limit1}(left). After applying gauge transformations at the two black interior vertices of $G$ and taking a limit, we obtain a weighted graph $(G',\wt')$ shown in \figref{fig:limit1}(right). The point $\Meas(G',\wt')$ belongs to the \emph{totally positive Grassmannian} (i.e., has all Pl\"ucker coordinates strictly positive). Yet, $\Meas(G',\wt')$ belongs to the \emph{boundary} of $\Ctnn_{2,4}$, i.e., to $\Ctnn_{2,4}\setminus\Ctp_{2,4}$.
\end{example}

A natural compactification of $\THtp_{k,n}$ taking into account the ratios of distances between pairs of colliding points in $\bv$ is the $(n-1)$-dimensional \emph{cyclohedron} $\Cyc_n$ studied in~\cite{BoTa,Simion}. See \cref{sec:aff_pos_cycloh_and_compact} for background.
 The cyclohedron $\Cyc_n$ may be obtained as the \emph{Axelrod--Singer} compactification~\cite{AxSi} of $\THtp_{k,n}$. In particular, the interior of $\Cyc_n$ is identified with $\THtp_{k,n}$.
\begin{theorem}
The map $\Measkn:\THtp_{k,n}\xrasim \Ctp_{k,n}$ extends to a continuous surjective map
\begin{equation*}%
  \Measbkn:\Cyc_n\to \Ctnn_{k,n}.
\end{equation*}
\end{theorem}
\noindent A similar result (\cref{thm:Measf:Cyc_to_Ctnn}) holds for arbitrary critical cells. Here, instead of the cyclohedron, one needs to take an \emph{affine poset cyclohedron} introduced in~\cite{crit_polyt}. For an arbitrary permutation $\fbi\in S_n$, the critical cell $\Ctp_{\fbi}$ is parameterized by a configuration space $\THtp_{\fbi}$ of $n$ points on a circle where some points are allowed to pass through each other. To this data, we associate an \emph{affine poset} $\Pafb$ such that the corresponding affine poset cyclohedron $\Cyc(\Pafb)$ gives a suitable compactification of $\THtp_{\fbi}$. This allows us to extend the boundary measurement map $\Measop_{\fbi}:\THtp_{\fbi}\to\Ctp_{\fbi}$ to a surjective continuous map
\begin{equation*}%
 \Measbf: \Cyc(\Pafb)\to\Ctnn_{\fbi}.
\end{equation*}
 By considering images of different faces of $\Cyc(\Pafb)$, we obtain a stratification of $\Ctnn_{\fbi}$. 

It turns out that the map $\Measbkn$ is far from a homeomorphism. Instead, it has the following remarkable property, which we call \emph{independence of infinitesimal ratios}. 
 Suppose that $\bv\pat\in \Ctp_{k,n}$ is a sequence of point configurations converging to some $\bv\in\Cyc_{n}$ as $t\to0$. Let $d\pat:=\max_{p,q\in[n]} |v\pat_p-v\pat_q|$. It turns out that for all $p,q$ such that $\lim_{t\to0}\frac{|v\pat_p-v\pat_q|}{d\pat}=0$, the limit $\Measbkn(\bv)$ of $\Measkn(\bv\pat)$ does not depend on distance ratios involving $|v\pat_p-v\pat_q|$. This property is surprising since the limiting edge weight function $\wt'$ \emph{does} depend on such distance ratios. However, the resulting limiting graph $G'$ is not reduced in general, and after applying \emph{reduction moves} (see \cref{fig:moves_red}) to it, all such ratios miraculously cancel each other out.

\begin{figure}
  \includegraphics[width=1.0\textwidth]{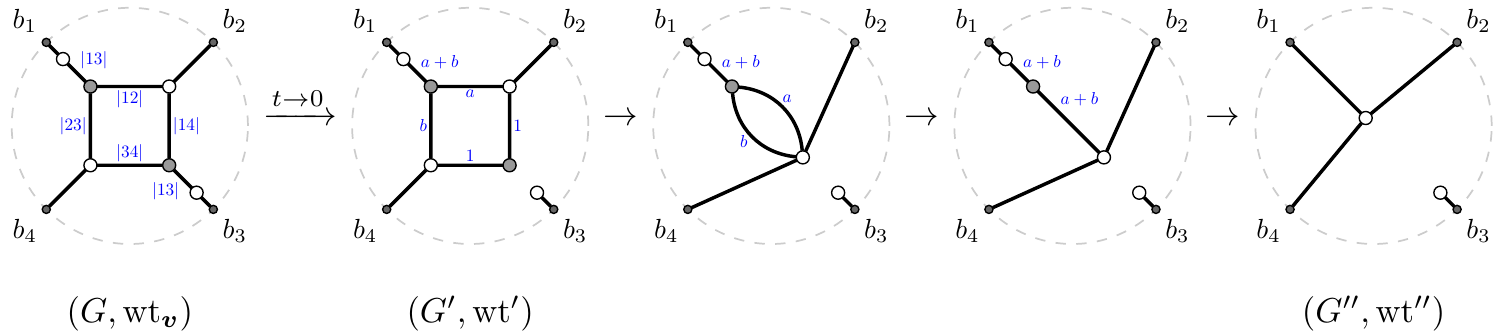}
  \caption{\label{fig:limit2} Taking a limit where the points $v_1,v_2,v_3$ collide but are far from $v_4$. After applying a sequence of reduction moves from \cref{fig:moves_red}, the edge weights involving relative distances between $v_1,v_2,v_3$ cancel out. See \cref{ex:intro:2} and \cref{thm:intro:hyp2n}.}
\end{figure}

\begin{example}\label{ex:intro:2}
Let $G$ be the graph in \cref{fig:plabic}, and suppose that $v_1,v_2,v_3,v_4$ collide so that
\begin{align*}%
  (|v_2-v_1|:|v_3-v_2|:|v_3-v_1|:|v_4-v_3|:|v_4-v_2|:|v_4-v_1|)&\to(0:0:0:1:1:1),\\
  (|v_2-v_1|:|v_3-v_2|:|v_3-v_1|)&\to(a:b:a+b),
\end{align*}
for some constants $a,b>0$. After applying gauge transformations and taking a limit, we obtain a weighted graph $(G',\wt')$ shown in \figref{fig:limit2}(middle left). Thus the edge weights $\wt'$ of $G'$ depend on the ratio $a:b$ in a non-trivial fashion. 
The graph $G'$ is not reduced, and after applying reduction moves to it, we see that all edge weights involving $a$ and $b$ cancel out; see \figref{fig:limit2}(right). Our result (\cref{thm:intro:hyp2n}) claims that this phenomenon occurs more generally for arbitrary $k$ and $n$, and for an arbitrary choice of the limiting ratios of distances between the points in $\bv$.
\end{example}

To explain independence of infinitesimal ratios formally, consider a map 
\begin{equation*}%
  \mapCD: \THtp_{k,n}\to \RP^{n-1},\quad \bv\mapsto (|v_2-v_1|:|v_3-v_2|:\cdots:|v_{n}-v_{n-1}|:|v_1-v_n|).
\end{equation*}
Passing to the closure, $\mapCD$ can be extended to a continuous map $\mapCD: \Cyc_{n}\to \RP^{n-1}.$ 
 The image $\mapCD(\Cyc_n)$ is essentially described by triangle inequalities, and it is straightforward to check (\cref{prop:delta2n}) that it may be identified with the \emph{second hypersimplex}
\begin{equation*}%
  \Delta_{2,n}:=\{(x_1,x_2,\dots,x_n)\in[0,1]^n\mid x_1+x_2+\cdots+x_n=2\}.
\end{equation*} 
\begin{theorem}[Independence of infinitesimal ratios]\label{thm:intro:hyp2n}
  The map $\Measbkn$ factors through the map $\mapCD$. That is, there exists a continuous map
\begin{equation*}%
  \mapDC:\Delta_{2,n}\to\Ctnnkn
\end{equation*}
making the following diagram commutative:
\begin{equation}\label{eq:intro:map_factors}
  \begin{tikzcd}[row sep=large]
\Cyc_n
\arrow[r,twoheadrightarrow,labels=below,"\mapCD"] 
\arrow[rr,twoheadrightarrow,bend left=20,labels=above,"\Measbkn"]
 & \Delta_{2,n} \arrow[r,dashed,"\mapDC",labels=below] & \Ctnnkn.
\end{tikzcd}
\end{equation}
Moreover, the map $\mapDC:\Delta_{2,n}\to\Ctnnkn$ is a stratification-preserving homeomorphism.
\end{theorem}
We note that currently we have no analog of \cref{thm:intro:hyp2n} for other critical cells. First, independence of infinitesimal ratios is very special to the top cell, and does not appear to hold for lower cells. Second, showing that the map $\mapDC$ is a homeomorphism relies on the \emph{injectivity conjecture}~\cite[Conjecture~4.3]{crit} being true for a certain family of critical cells; see \cref{sec:inj}. This conjecture remains wide open for arbitrary critical cells. Nevertheless, limited computational evidence suggests that the stratified space $\Ctnn_f$ may be polytopal for each $f\in S_n$.

\subsection*{Acknowledgments} The author is grateful to the anonymous referee for their valuable feedback on the first version of the manuscript.

\section{Background on critical varieties}
We review the background on \emph{positroid cells} inside the \emph{totally nonnegative Grassmannian}~\cite{Pos}; see also~\cite{LamCDM}. We then recall the construction of \emph{critical cells} introduced in~\cite{crit}.

\subsection{Planar bipartite graphs}\label{sec:backr:plabic}
Fix a planar graph $G$ as in \cref{sec:intro:main}. Recall that the $n$ boundary vertices of $G$ are assumed to be black and to have degree $1$, and that $G$ is assumed to be reduced. Any non-reduced graph $G$ may be transformed into a reduced one using the moves in \cref{fig:moves_red,fig:moves_cs}.

\begin{figure}
  \includegraphics[width=0.9\textwidth]{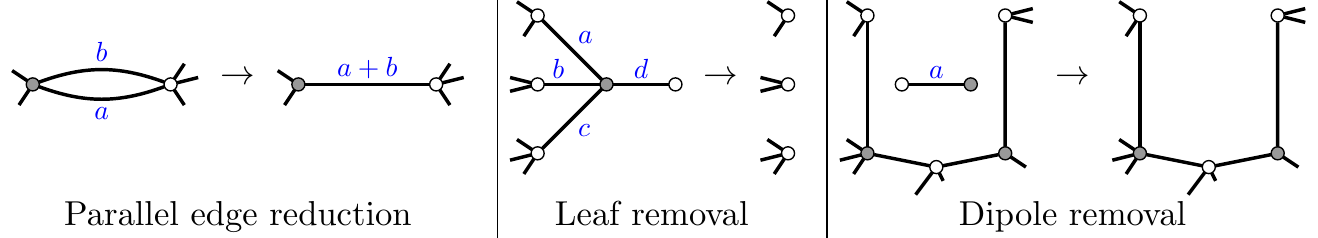}
  \caption{\label{fig:moves_red} Reduction moves for planar bipartite graphs. Each move preserves the boundary measurements.}
\end{figure}

\begin{figure}
  \includegraphics[width=1.0\textwidth]{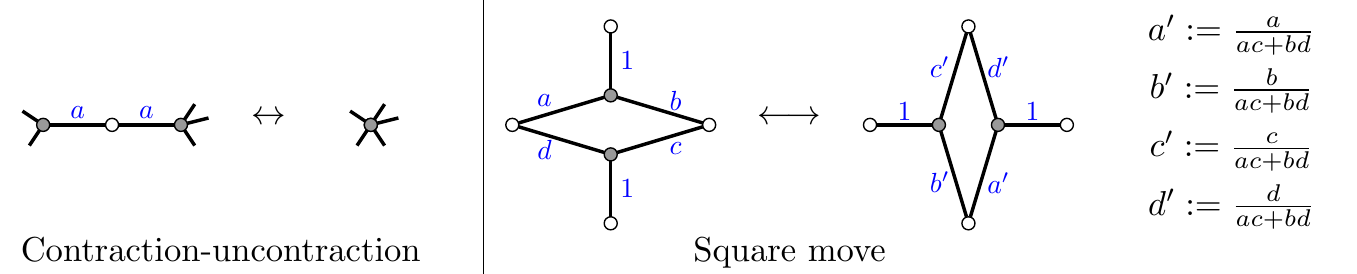}
  \caption{\label{fig:moves_cs} Moves for planar bipartite graphs preserving the boundary measurements and the strand permutation.}
\end{figure}

We switch to denoting strand permutations by $\fb_G$, and reserve the notation $f_G$ for \emph{bounded affine permutations} introduced below.
\begin{definition}\label{dfn:BAP}
A \emph{$(k,n)$-bounded affine permutation} is a bijection $f:\Z\to\Z$ such that
\begin{itemize}
\item $f(j+n)=f(j)+n$ for all $j\in\Z$,
\item $\sum_{j=1}^{n}(f(j)-j)=kn$, and
\item $j\leq f(j)\leq j+n$ for all $j\in\Z$.
\end{itemize}
\end{definition}
We denote the set of $(k,n)$-bounded affine permutations by $\Bkn$. For $f\in\Bkn$, we let $\fb\in S_n$ be obtained by reducing $f$ modulo $n$. In other words, $\fb$ is uniquely determined by the conditions $\fb(j)\in[n]$ and $\fb(j)\equiv f(j)$ modulo $n$ for all $j\in[n]$.

\begin{remark}\label{rmk:loopless}
We say that $f\in\Bkn$ is \emph{loopless} if $f(j)\neq j$ for all $j\in\Z$. Each permutation $\fb\in S_n$ arises via the above procedure from a unique loopless bounded affine permutation $f\in \Bkn$: for $j\in[n]$, one sets $f(j):=\fb(j)$ if $\fb(j)>j$ and $f(j):=\fb(j)+n$ otherwise. The remaining values $f(j+dn)=f(j)+dn$ are automatically determined for all $d\in\Z$. Positroid cells are labeled by arbitrary bounded affine permutations while critical cells are labeled by loopless bounded affine permutations, which is why in the introduction we used permutations in $S_n$ to label critical cells.
\end{remark}

In general, the bounded affine permutation $f_G$ is recovered from $\fb_G$ as follows. For $j\in[n]$, if $\fb_G(j)\neq j$ then $f_G(j)$ is uniquely determined by the conditions $j\leq f_G(j)\leq j+n$ and $f_G(j)\equiv \fb_G(j)$ modulo $n$. 
If $\fb_G(j)=j$ then, depending on the structure of $G$ (see \cref{dfn:loop_coloop}), either $j$ is a \emph{loop} (i.e., $f_G(j)=j$) or $j$ is a \emph{coloop} (i.e., $f_G(j)=j+n$). 

An \emph{affine inversion} of $f\in\Bkn$ is a pair $(p,q)\in\Z^2$ such that $p<q$ and $f(p)>f(q)$. The \emph{length} $\ell(f)$ of $f$ is the number of affine inversions of $f$ considered modulo $n$:
\begin{equation*}
  \ell(f):=\#\{p,q\in\Z\mid p<q,\ f(p)>f(q),\text{ and }p\in[n]\}.
\end{equation*}

The \emph{(real) Grassmannian} $\Gr(k,n)$ is the set of all linear $k$-dimensional subspaces of $\R^n$. Choosing a basis of each subspace, $\Gr(k,n)$ may be identified with the space of full rank $k\times n$ matrices $M$ considered modulo row operations. With this identification, one has a collection of \emph{Pl\"ucker coordinates} on $\Gr(k,n)$. Let ${[n]\choose k}$ denote the set of $k$-element subsets of $[n]$, and for each $I\in{[n]\choose k}$ and a $k\times n$ matrix $M$ we let $\Delta_I(M)$ denote the maximal minor of $M$ with column set $I$. Letting $I$ vary, we obtain the \emph{Pl\"ucker embedding} $\Gr(k,n)\hookrightarrow \RP^{{n \choose k}-1}$ sending the row span of $M$ to $(\Delta_I(M))_{I\in{[n]\choose k}}\in \RP^{{n \choose k}-1}$.

Let $\RPtp^{r-1}$ be the subset of $\RP^{r-1}$ where all coordinates are nonzero and have the same sign, and let $\RPtnn^{r-1}$ be the closure of $\RPtp^{r-1}$. The \emph{totally nonnegative Grassmannian} $\Grtnn(k,n)$ is the subset of $\Gr(k,n)$ where all nonzero Pl\"ucker coordinates have the same sign. In other words, $\Grtnn(k,n)$ is the preimage of $\RPtnn^{{n \choose k}-1}$ under the Pl\"ucker embedding.

Given a planar bipartite graph $G$ as above, the \emph{boundary measurement map} $\Measop_G:\Rtp^{E(G)}\to\Grtnn(k,n)$ is defined using the \emph{dimer model} on $G$. An \emph{almost perfect matching} $\Acal$ of $G$ is a collection of edges of $G$ which uses each interior vertex exactly once. Importantly (cf. \cref{lemma:Gr_limit} below), in order to define the boundary measurement map $\Measop_G$, we assume that $G$ admits at least one almost perfect matching.

Recall that the boundary vertices of $G$ are assumed to be black and have degree $1$. For an almost perfect matching $\Acal$, let $\partial(\Acal)\subseteq[n]$ denote the set of $p\in[n]$ such that the boundary vertex $b_p$ is used by $\Acal$. There is an integer $0\leq k\leq n$ depending only on $G$ such that $|\partial(\Acal)|=k$ for any almost perfect matching $\Acal$ of $G$. Given an edge weight function $\wt:E(G)\to\Rtp$, the weight $\wt(\Acal):=\prod_{e\in \Acal} \wt(e)$ of $\Acal$ is the product of the weights of the edges used by $\Acal$. For $I\in{[n]\choose k}$, we set
\begin{equation*}%
  \Delta_I(G,\wt):=\sum_{\Acal:\ \partial(\Acal)=I} \wt(\Acal).
\end{equation*}
We view the resulting boundary measurements
\begin{equation}\label{eq:Meas_dfn}
  \Meas(G,\wt):=(\Delta_I(G,\wt))_{I\in{[n]\choose k}}
\end{equation}
 up to multiplication by a common scalar, i.e., as an element of $\RP^{{n\choose k}-1}$. It was shown in~\cite{Pos,Talaska} (see~\cite[Theorem~4.1]{Lam}) that the entries of $\Meas(G,\wt)$ are the Pl\"ucker coordinates of some point of $\Grtnn(k,n)$ which we also denote by $\Meas(G,\wt)$.

\begin{definition}\label{dfn:loop_coloop}
It is known that when $\fb_G(j)=j$, exactly one of the following holds:
\begin{itemize}
\item $j\notin\Acal$ for any almost perfect matching $\Acal$ of $G$;
\item $j\in\Acal$ for any almost perfect matching $\Acal$ of $G$.
\end{itemize}
In the former case, we say that $j$ is a \emph{loop} and set $f_G(j)=j$. In the latter case, we say that $j$ is a \emph{coloop} and set $f_G(j)=j+n$. This completes the definition of the bounded affine permutation $f_G\in\Bkn$ associated to $G$. For $f\in\Bkn$, we let $\Gred(f)$ denote the set of all reduced planar bipartite graphs $G$ satisfying $f_G=f$. For  $G\in\Gred(f)$, the \emph{positroid cell} $\Ptp_G:=\{\Meas(G,\wt)\mid\wt:E(G)\to\Rtp\}$ depends only on $f$ and is denoted $\Ptp_f$. The \emph{top cell} bounded affine permutation $\fkn\in\Bkn$ is defined by $\fkn(p)=p+k$ for all $p\in\Z$. 
\end{definition}

\subsection{Critical cells}
Let $f\in\Bkn$ be a loopless bounded affine permutation and let $\fb\in S_n$ be the corresponding permutation. The combinatorics of the critical cell $\Ctp_f$ associated to $f$ is described by the following objects.
\begin{definition}\label{dfn:strand}
Place $2n$ points $\m 1,\p 1,\dots,\m n,\p n$ on the circle in clockwise order. The \emph{reduced strand diagram} of $f$ is obtained by drawing an arrow $\p s\to \m{\fb(s)}$ for each $s\in[n]$. We say that $p,q\in[n]$, $p\neq q$, \emph{form an $f$-crossing} if the arrows $\p s\to\m p$ and $\p t\to\m q$ cross, where $s:=\fb^{-1}(p)$ and $t:=\fb^{-1}(q)$. We say that $f$ has a \emph{connected strand diagram} if the resulting union of $n$ arrows is topologically connected. See \figref{fig:strand}(left) for an example.
\end{definition}
Throughout the paper, we assume that $f$ has a connected strand diagram. When the strand diagram of $f$ is not connected, the corresponding critical cell $\Ctp_f$ (as well as its closure $\Ctnn_f$) factorizes as a product over its connected components; see~\cite[Section~4.4]{crit}.

\begin{definition}
A tuple $\bth=(\th_1,\th_2,\dots,\th_n)\in\R^n$ is called \emph{$f$-admissible} if whenever two indices $1\leq p<q\leq n$ form an $f$-crossing, we have
\begin{equation}\label{eq:f_adm_dfn}
  \th_p<\th_q<\th_p+\pi.
\end{equation}
We let
\begin{equation}\label{eq:THtp_dfn}
  \THtp_f:=\{\bth\in\R^n\mid \th_1=0\text{ and $\bth$ is $f$-admissible}\}.
\end{equation}
\end{definition}
Letting $v_r:=\exp(2i\th_r)$ for $r\in[n]$, we obtain a configuration $\bv=(v_1,v_2,\dots,v_n)$ of $n$ points on the unit circle which are not necessarily distinct or ordered counterclockwise. The condition $\th_1=0$ reflects that we consider these points modulo rotations of the circle.

A graph $G\in\Gred(f)$ is called \emph{contracted} if it has no degree $2$ vertices that are not adjacent to the boundary. Any graph $G\in\Gred(f)$ may be transformed into a contracted one using \emph{contraction-uncontraction moves} (\figref{fig:moves_cs}(left)) which do not affect the boundary measurements of $G$.

Given a contracted graph $G\in\Gred(f)$ and an $f$-admissible tuple $\bth\in\THtp_f$, we define a weight function $\wt_\bth:E(G)\to\Rtp$ similarly to~\eqref{eq:intro:wt_bv}: if $e\in E(G)$ is labeled by $\{p,q\}$ with $1\leq p<q\leq n$ then we set
\begin{equation}\label{eq:wt_bth}
  \wt_\bth(e):=  \begin{cases}
    \sin(\th_q-\th_p), &\text{if $e$ is not incident to a boundary vertex;}\\
    1, &\text{otherwise.}
  \end{cases}
\end{equation}
By~\cite[Proposition~4.2]{crit}, we indeed get $\wt_{\bth}(e)>0$ for all $e\in E(G)$. Setting $v_r:=\exp(2i\th_r)$ for $r\in[n]$, we get $\sin(\th_q-\th_p)=\frac12|v_q-v_p|$. Thus $\wt_\bth$ differs from $\wt_\bv$ defined in~\eqref{eq:intro:wt_bv} by applying gauge transformations at all black interior vertices.

The crucial property of this assignment of edge weights is that the resulting boundary measurements are invariant under \emph{square moves} (\figref{fig:moves_cs}(right)). Thus it follows from the results of~\cite{Pos} that the point $\Meas(G,\wt_\bth)$ does not depend on the choice of $G$. We denote $\Meas(f,\bth):=\Meas(G,\wt_\bth)$. The \emph{critical cell} is given by
\begin{equation*}%
  \Ctp_f:=\{\Meas(f,\bth)\mid \bth\in\R^n\text{ is $f$-admissible}\}.
\end{equation*}

\section{Affine poset cyclohedra}
We review some definitions and properties of affine posets and the associated polytopes; see~\cite{crit_polyt} for further details. 

\subsection{Order polytopes and tubings}\label{sec:order_tubings}
We start with ordinary posets. Let $(P,\leqp)$ be a connected (i.e., having a connected Hasse diagram) poset with $|P|\geq2$. 
Let $\al_P:\R^P\to\R$ be a linear function given by
\begin{equation*}%
  \al_P(\bx):=\sum_{p\,\precdot_P\, q} x_q-x_p,
\end{equation*}
where the sum is taken over all covering relations $p\precdot_P q$ in $P$. Let $\RSZ^P$ denote the linear subspace of $\R^P$ consisting of vectors whose sum of coordinates is zero. Consider a $(|P|-2)$-dimensional polytope %
\begin{equation*}%
  \Ord(P):=\{\bx\in\RSZ^P\mid \al_P(\bx)=1\text{ and }x_p\leq x_q\text{ for all $p\leqp q$}\}.
\end{equation*}
 When $P$ has a maximal and a minimal element, $\Ord(P)$ is projectively equivalent to the \emph{order polytope}~\cite{Stanley_two} of $P$; see~\cite[Remark~2.5]{crit_polyt}.

For a subset $\t\subseteq P$, we say that $\t$ is \emph{convex} if for any three elements $p\leqp q\leqp r$ such that $p,r\in\t$, we have $q\in\t$. We say that $\t$ is \emph{connected} if the restriction of $\leqp$ to $\t$ is a connected poset. A \emph{$P$-tube} is a convex connected nonempty subset $\t\subseteq P$. A \emph{tubing partition} of $P$ is a set partition $\T$ of $P$ into disjoint $P$-tubes such that the directed graph $D_\T$ with vertex set $V(D_\T):=\T$ and edge set
\begin{equation}\label{eq:acyclic}
  E(D_\T):=\{(\t,\t')\mid \t\cap \t'=\emptyset\text{ and }p\lp q \text{ for some $p\in\t$, $q\in\t'$}\}
\end{equation}
is acyclic. The faces of $\Ord(P)$ are in bijection with tubing partitions of $P$. Explicitly, given a point $\bx\in\Ord(P)$, consider a maximal by inclusion set $I\subseteq P$ such that all coordinates in $\{x_p\}_{p\in I}$ coincide. Then $I$ is a disjoint union of $P$-tubes, which are the connected components of the induced subgraph of the Hasse diagram of $P$ with vertex set $I$. Collecting these $P$-tubes for all such sets $I$, we obtain a tubing partition of $P$ denoted $\tubes(\bx)$.

\begin{definition}
An \emph{affine poset (of order $n\geq1$)}  is a poset $\Pa=(\Z,\leqPa)$ such that:
\begin{itemize}%
\item for all $p\in\Z$, $p\lPa p+n$;
\item for all $p,q\in\Z$, $p\leqPa q$ if and only if  $p+n\leqPa q+n$;%
\item for all $p,q\in\Z$, we have $p\leqPa q+dn$ for some $d\geq0$.
\end{itemize}
\end{definition}
\noindent We denote $|\Pa|:=n$.

We identify points $\bth\in\R^{|\Pa|}$ with infinite sequences $\btht=(\tht_p)_{p\in\Z}$ satisfying $\tht_p=\th_p$ for $p\in[n]$ and $\tht_{p+n}=\tht_p+\pi$ for $p\in\Z$. Consider the $(n-1)$-dimensional \emph{affine order polytope} $\Orda(\Pa)$ and its interior $\Oao(\Pa)$ defined by
\begin{align}
  \Orda(\Pa)&:=\{\bth\in\R^{|\Pa|}\mid \th_1=0\text{ and }\tht_p\leq\tht_q\text{ for all $p\leqpa q$}\},\\
\label{eq:Oao_Pa}
  \Oao(\Pa)&:=\{\bth\in\R^{|\Pa|}\mid \th_1=0\text{ and }\tht_p<\tht_q\text{ for all $p\lpa q$}\}.
\end{align}
A \emph{$\Pa$-tube} (or simply a \emph{tube}) is a convex connected nonempty subset $\t\subseteq\Pa$ such that either $\t=\Pa$ or $\t$ contains at most one element in each residue class modulo $n$. For each tube $\t$, we denote by $\eq[\t]:=\{\t+dn\mid d\in\Z\}$ its \emph{equivalence class}, where $\t+dn:=\{p+dn\mid p\in\t\}$. A collection $\T$ of tubes is called \emph{$n$-periodic} if it is a union of such equivalence classes. 

We say that two sets $A,B$ are \emph{nested} if either $A\subseteq B$ or $B\subseteq A$. 
\begin{definition}
A \emph{$\Pa$-tubing} (or simply a \emph{tubing}) is an $n$-periodic collection $\T$ of tubes such that any two tubes in $\T$ are either nested or disjoint, and such that the directed graph $D_\T$ given by~\eqref{eq:acyclic} is acyclic. A tube $\t$ is called \emph{proper} if $\t\neq\Pa$ and $|\t|>1$. A tubing $\T$ is called \emph{proper} if it consists of proper tubes. A \emph{tubing partition} of $\Pa$ is a tubing $\T$ which is simultaneously a set partition of $\Z$.
\end{definition}
The face poset of $\Orda(\Pa)$ is isomorphic to the poset of tubing partitions of $\Pa$ ordered by refinement. For example, the vertices of $\Orda(\Pa)$ are in bijection with equivalence classes of \emph{maximal proper tubes} which are tubes $\t\neq \Pa$ satisfying $|\t|=n$. For a point $\bth\in\Orda(\Pa)$, we let $\tubes(\bth)$ denote the corresponding tubing partition of $\Pa$.

\begin{example}\label{ex:circ}
Let $n=5$. Consider the affine poset $\Pa$ of order $|\Pa|=n$ in \figref{fig:circ}(a). We may identify $\Orda(\Pa):=\{(\th_2,\th_3,\th_4,\th_5)\in\R^4\mid 0\leq \th_3\leq \th_4\leq \pi\text{ and } 0\leq \th_2\leq\th_5\leq\pi\}$. Thus, the order polytope $\Orda(\Pa)$ is the direct product of two triangles. The tubing $\T$ shown in \figref{fig:circ}(b) consists of the tubes $\tube:=\{0,1,2,3,4\}$, $\tube':=\{0,1,2\}$, $\tube'':=\{3,4\}$, and the tubes equivalent to them. The tube $\tube$ is a maximal proper tube; the corresponding vertex of $\Orda(\Pa)$ is given by $\th_2=\th_3=\th_4=0$, $\th_5=\pi$. This vertex is the limit inside $\Orda(\Pa)$ of the family $\bth^{(t)}$ of points of $\Oao(\Pa)$ shown in \figref{fig:circ}(c). Here, $\th_2^{(t)}=at^2$, $\th_3^{(t)}=at^2+t$, $\th_4^{(t)}=at^2+t+ct^2$, and $\th_5^{(t)}=-bt^2$. A more refined limit will be considered in \cref{ex:circ2}.
\end{example}

\begin{figure}
\begin{tabular}{c|c|ccc}
\includegraphics[width=0.1\textwidth]{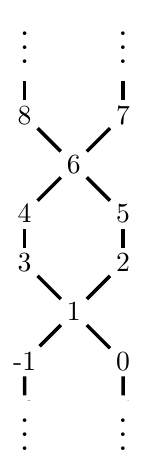}
& 
\includegraphics[width=0.1\textwidth]{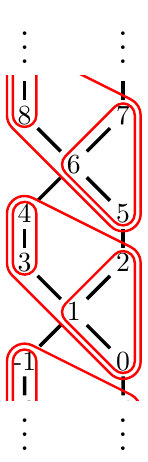}
&
\includegraphics[width=0.24\textwidth]{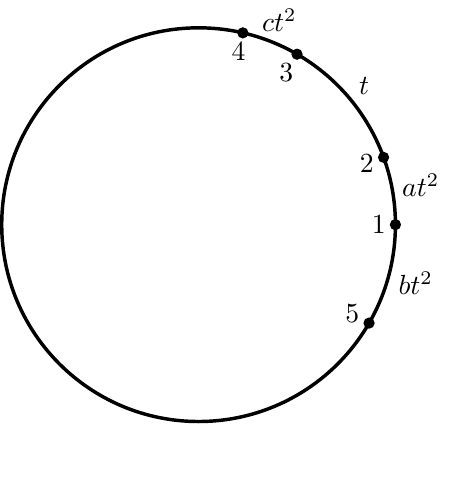}
&
\begin{tikzpicture}[baseline=(Z.base)]
\coordinate(Z) at (0,-2.7);
\node(A) at (0,0){$\xrightarrow{t\to0}$};
\end{tikzpicture}
&
\includegraphics[width=0.4\textwidth]{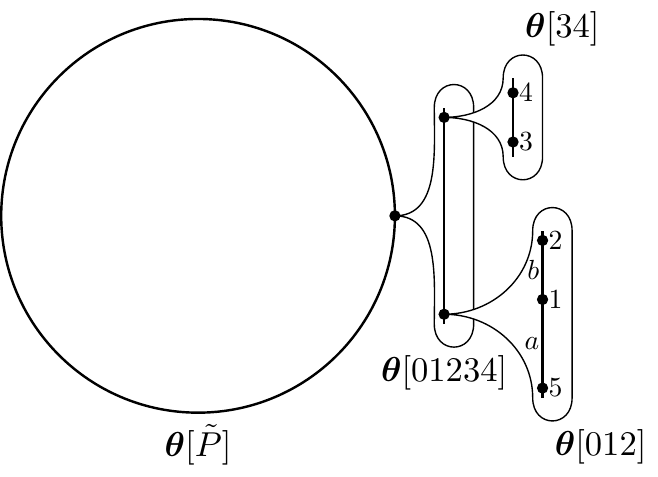}
\\
  (a) $\Pa$ & (b) $\T$ & (c) $\bth^{(t)}\in\Oao(\Pa)$ &  & (d) $\bth\in\Comp(\Pa)$
\end{tabular}
  \caption{\label{fig:circ}An affine poset $\Pa$, a tubing $\T$ of $\Pa$, and a family $\bth^{(t)}\in\Oao(\Pa)$ converging to a point $\bth$ in the compactification $\Comp(\Pa)$ satisfying $\T(\bth)=\T$. Here $a,b,c>0$ are constants, and the limiting point $\bth$ depends on the ratio $a:b$ but does not depend on $c$. See \cref{ex:circ,ex:circ2}.}
\end{figure}
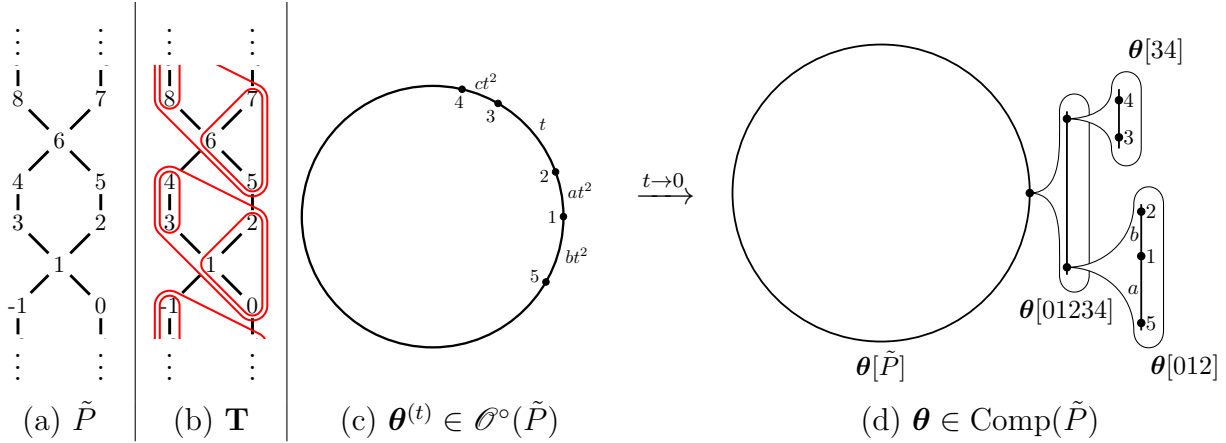

\subsection{Affine poset cyclohedra and compactifications}\label{sec:aff_pos_cycloh_and_compact}
We showed in~\cite{crit_polyt} that there is an $(n-1)$-dimensional polytope $\Cyc(\Pa)$, called an \emph{affine poset cyclohedron}, whose face poset is the poset of proper tubings ordered by reverse inclusion. For example, the vertices of $\Cyc(\Pa)$ are in bijection with proper tubings $\T$ satisfying $|\EQ[\T]|=n-1$, where $\EQ[\T]:=\{\eq[\t]\mid \t\in\T\}$ is the set of equivalence classes of tubes in $\T$. 


In addition, we showed in~\cite{crit_polyt} that $\Cyc(\Pa)\cong \Comp(\Pa)$ arises as a compactification of the space $\Oao(\Pa)$. We first explain the construction of $\Comp(\Pa)$ informally. The space $\Oao(\Pa)$ defined in~\eqref{eq:Oao_Pa} may be identified with a configuration space of $n$ points on a circle: setting $v_r:=\exp(2i\th_r)$ for $r\in\Z$ as in \cref{rmk:exp_simplex}, we have $v_{r+n}=v_r$ for all $r\in\Z$. The points $v_p,v_q$ are not allowed to pass through each other whenever $p,q\in\Z$ are comparable in $\Pa$. For instance, for $\Pa$ in \figref{fig:circ}(a), $v_3$ cannot pass through $v_4$ but can pass through $v_2$. The compactification $\Comp(\Pa)$ is obtained by allowing the points to collide and keeping track of the ratios of distances between the points in the limit. This leads to a recursive picture of the type shown in \figref{fig:circ}(d). The positions of the points on the circle define a point $\bth[\Pa]\in\Ord(\Pa)$. The face of $\Ord(\Pa)$ containing $\bth[\Pa]$ is labeled by a tubing partition $\tubes(\bth[\Pa])$ of $\Pa$. For each $\tube\in\tubes(\bth[\Pa])$, the points in $\tube$ have collided together. However, we would like to ``zoom in'' and keep track of the ratios of distances between these points, which naturally gives rise to a point in $\Ord(\tube)$ denoted $\bth[\tube]$. Iterating this process, we obtain a tubing $\T:=\T(\bth)$ (cf. \cref{dfn:T(bth)} below) and a collection $(\bth[\t])_{\t\in\T\sqcup\{\Pa\}}$, where  $\bth[\t]\in\Ord(\t)$ for each $\t\in\T\sqcup\{\Pa\}$. 
 Keeping track of this data while letting the points collide in all possible ways, we obtain the compactification $\Comp(\Pa)$.

\begin{example}\label{ex:circ2}
Consider a sequence of points $\bth^{(t)}\in\Oao(\Pa)$  given in \figref{fig:circ}(c); cf. \cref{ex:circ}. Taking a limit as $t\to0$, we find that $(|v_1-v_2|:|v_1-v_5|:|v_2-v_5|)\to(a:b:a+b)$, $(|v_1-v_2|:|v_2-v_3|)\to(0:1)$, and $(|v_2-v_3|:|v_3-v_4|)\to(1:0)$. The resulting tubing $\T(\bth)$ is shown in \figref{fig:circ}(b). Let $\tube:=\{0,1,2,3,4\}$, $\tube':=\{0,1,2\}$, $\tube'':=\{3,4\}$. The point $\bth[\Pa]$ corresponds to the vertex of $\Ord(\Pa)$ labeled by the maximal proper tube $\tube=\{0,1,2,3,4\}$; cf. \cref{ex:circ}. On the other hand, the point $\bth[\tube']$ for $\tube'=\{0,1,2\}$ records the $(a:b:a+b)$ ratio of distances between $v_5,v_1,v_2$. Note that the limit $(|v_1-v_2|:|v_3-v_4|)\to(a:c)$ is not recorded by the points $\bth[\Pa],\bth[\tube],\bth[\tube'],\bth[\tube'']$ shown in \figref{fig:circ}(d); the value of the constant $c$ is lost in the limit. 
\end{example}


We now define $\Comp(\Pa)$ formally. Let $\t\subsetneq\Pa$ be a proper tube. We treat $\t$ as a finite subposet $(\t,\leqpa)$ of $\Pa$, thus, we may consider the order polytope $\Ord(\t)$. The projection $\R^{|\Pa|}\to\R^\t$ sending $(\tht_p)_{p\in\Z}\mapsto(\tht_p)_{p\in\t}$ gives rise to a map $\res_\t:\Oao(\Pa)\to\Oo(\t)$. (We will later obtain the point $\bth[\t]\in\Ord(\t)$ as the limit of $\res_\t(\bth^{(t)})$ as $t\to0$.) More precisely, given any set $A\supseteq \t$, define the following maps:
\begin{align*}%
 &\avg_\t:\R^A\to \R, \quad \bx\mapsto \frac1{|\t|}\sum_{p\in \t} x_p; &\quad &\psz^\t:\R^A\to\RSZ^\t,  \quad \bx\mapsto (x_p-\avg_\t(\bx))_{p\in \t};\\
  &\al_\t:\R^A\to\R,\quad \bx\mapsto \sum_{p,q\in\t:\ p\,\precdot_{\Pa}\, q} x_q-x_p;&\quad &\res_\t:\R^A\dashrightarrow \R^\t,\quad \bx\mapsto \frac1{\al_\t(\bx)} \psz^\t(\bx).
\end{align*}
Here $\res_\t$ is a rational map defined on the subset of $\R^A$ where $\al_\t(\bx)\neq0$. Applying this construction to the case $A=\Z$, we obtain a map $\res_\t:\Oao(\Pa)\to\Oo(\t)$. Notice that $\al_\t$ takes strictly positive values on $\Oao(\Pa)$. By convention, for $\bth\in\Oao(\Pa)$, we set $\res_{\Pa}(\bth):=\bth$. Let
\begin{equation*}%
    \ResPa: \Oao(\Pa)\to\bar{\prod_{|\t|>1}} \Ord(\Pax[\t]),\quad \bth\mapsto(\res_\t(\bth))_{|\t|>1}.
\end{equation*}
Here $\prodb_{|\t|>1} \Ord(\Pax[\t])$ is the set of points $(\bth[\t])_{|\t|>1}\in\prod_{|\t|>1} \Ord(\Pax[\t])$ satisfying $\bth[\t]=\bth[\t']$ whenever two tubes $\t,\t'$ are equivalent. The product is taken over all non-singleton tubes $\t$, including the case $\t=\Pa$.  The compactification
\begin{equation}\label{eq:intro:Comp_Pa}
    \Comp(\Pa):=\overline{\ResPa(\Oao(\Pa))}
\end{equation}
is defined as the closure of the image of $\ResPa$.

By definition, each point $\bth\in\Comp(\Pa)$ is an element $(\bth[\t])_{|\t|>1}$ of the product $\prodb_{|\t|>1} \Ord(\Pax[\t])$. We refer to its coordinates as $(\tht_i[\t])_{i\in\t}$ for each non-singleton tube $\t$. We showed in~\cite[Proposition~3.9]{crit_polyt} that $\Comp(\Pa)$ may be alternatively described as the subset of $\prodb_{|\t|>1} \Ord(\Pax[\t])$ consisting of all points satisfying the following \emph{coherence condition}:
\begin{equation}\label{eq:coh}
\text{for any $\t\subsetneq\t_+$ with $|\t|>1$, there exists $\la\in\Rtnn$ such that }\psz^\t(\bth\tbr[\tp])=\la\bth\tbr[\t].
\end{equation}
\begin{definition}\label{dfn:T(bth)}
For $\bth\in\Comp(\Pa)$, let $\HAT(\bth)$ be the smallest collection of tubes such that
\begin{itemize}
\item $\HAT(\bth)$ contains $\Pa$;
\item for each non-singleton $\t\in\HAT(\bth)$, $\HAT(\bth)$ also contains all tubes in $\tubes(\bth[\t])$.
\end{itemize}
We let $\T(\bth)$ be obtained from $\HAT(\bth)$ by removing $\Pa$ and all singleton tubes. More generally, for a proper tubing $\T$, we let $\HAT$ be obtained from $\T$ by adding $\Pa$ and all singleton tubes, and vice versa.
\end{definition}

\begin{remark}
The informal description given at the beginning of the section resulted in a pair $(\T,(\bx[\t])_{\t\in\T\sqcup\{\Pa\}})$. In the formal description~\eqref{eq:intro:Comp_Pa}, a point $\bx\in\Comp(\Pa)$ is by definition a collection $\bx[\t]$ for all tubes $\t$ satisfying $|\t|>1$. \Cref{dfn:T(bth)} explains how to recover the tubing $\T$ from $\bx\in\Comp(\Pa)$. Moreover, one can see from~\eqref{eq:coh} that for each $\t'\notin\T\sqcup\{\Pa\}$, the point $\bx[\t']$ is uniquely determined by the tuple $(\bx[\t])_{\t\in\T\sqcup\{\Pa\}}$; see~\cite[Proposition~3.11]{crit_polyt}. This explains the equivalence between the formal description~\eqref{eq:intro:Comp_Pa} and the informal description above.
\end{remark}

The space $\Comp(\Pa)$ is naturally subdivided into cells labeled by proper tubings: for a proper tubing $\T$, the corresponding cell is given by
\begin{equation*}%
  \DT:=\{\bth\in\Comp(\Pa)\mid \T(\bth)=\T\}.
\end{equation*}
Cell closure relations are given by reverse inclusion of tubings: 
\begin{equation*}%
  \Closure(\D_\Tubing)=\bigsqcup_{\Tubing'\supseteq\Tubing} \D_{\Tubing'}.
\end{equation*}
\begin{theorem}[{\cite[Theorem~1.11]{crit_polyt}}]\label{thm:cyc_comp}
    There exists a stratification-preserving homeomorphism $\Cyc(\Pa)\xrasim\Comp(\Pa)$.
\end{theorem}
\begin{remark}\label{rmk:interior_identify}
In what follows, we always identify $\Cyc(\Pa)$ with $\Comp(\Pa)$. The map $\ResPa$ gives a homeomorphism between $\Oao(\Pa)$ and the unique open dense cell $\D_{\emptyset}$ of $\Comp(\Pa)$, and we identify each of these spaces with the interior of the affine poset cyclohedron: 
\begin{equation*}
  \Oao(\Pa)\cong\D_{\emptyset}\cong\Cyco(\Pa).
\end{equation*}
\end{remark}

\subsection{Circular chains}
Let $\Pa$ be an affine poset. Our goal is to construct a particular family of continuous functions on $\Cyc(\Pa)$ indexed by \emph{circular $\Pa$-chains}.

\begin{definition}\label{dfn:circular_chain}
We say that a tuple $\pch:=(p_1,p_2,\dots,p_r)$ of integers is a \emph{circular $\Pa$-chain} if 
\begin{equation}\label{eq:vertex_chain}
  p_1\lPa p_2\lPa\cdots\lPa p_r \lPa p_1+n.
\end{equation}
Thus $\pch$ is a circular $\Pa$-chain if and only if $\sigma(\pch):=(p_2,\dots,p_r,p_1+n)$ is a circular $\Pa$-chain. We say that two such tuples differ by \emph{cyclic relabeling}. We say that a tube $\t$ \emph{contains the residues of $\pch$ modulo $n$} if for each $j\in[r]$, we have $p_j+d_jn\in\t$ for some $d_j\in\Z$. Equivalently, since each tube $\t$ is convex, it follows that $\t$ contains the residues of $\pch$ modulo $n$ if and only if $\t$ contains all elements of a circular $\Pa$-chain $\sigma^s(\pch)$ for some $s\in\Z$.
\end{definition}

Given a circular $\Pa$-chain $\pch=(p_1,p_2,\dots,p_r)$ and a point $\bth\in\Cyc(\Pa)$, the point $\bth[\Pa]\in\Orda(\Pa)$ satisfies
\begin{equation}\label{eq:tht_leq_dots_leq}
  \tht_{p_1}[\Pa]\leq \tht_{p_2}[\Pa]\leq\dots\leq \tht_{p_r}[\Pa]\leq \tht_{p_1+n}[\Pa]=\tht_{p_1}[\Pa]+\pi.
\end{equation}
For any tube $\t\subsetneq\Pa$ satisfying $p_1,p_2,\dots,p_r\in\t$, the vector $\bth[\t]\in\Ord(\Pax[\t])$ satisfies
\begin{equation}\label{eq:th_leq_dots_leq}
  \tht_{p_1}[\t]\leq \tht_{p_2}[\t]\leq\dots\leq \tht_{p_r}[\t].
\end{equation}

\begin{lemma}\label{lemma:sin_map}
Let $\Pa$ be an affine poset, and suppose that $\pch=(p_1,p_2,\dots,p_r)$ is a circular $\Pa$-chain. Then the map
\begin{equation}\label{eq:sin_map}
  \sinmap_{\pch}: \Cyco(\Pa) \to \RPtp^{r-1},\quad \bth \mapsto \left(\sin(\tht_{p_2}-\tht_{p_1}):\cdots:\sin(\tht_{p_r}-\tht_{p_{r-1}}):\sin(\tht_{p_1+n}-\tht_{p_r})\right)
\end{equation}
extends to a continuous map 
\begin{equation*}%
  \sinmapb_{\pch}: \Cyc(\Pa) \to \RPtnn^{r-1}.
\end{equation*}
\end{lemma}
\begin{proof}
Let $\bth\in \Comp(\Pa)\cong \Cyc(\Pa)$ and let $\T:=\T(\bth)$ be the associated tubing. Let $\t\in\HAT$ be a minimal by inclusion tube containing the residues of $\pch$ modulo $n$. 

If $\t=\Pa$ then we set
\begin{equation}\label{eq:ratmap_Pa}
\sinmapb_\pch(\bth):=\left(\sin(\tht_{p_2}[\Pa]-\tht_{p_1}[\Pa]):\cdots:\sin(\tht_{p_r}[\Pa]-\tht_{p_{r-1}}[\Pa]): \sin(\tht_{p_1+n}[\Pa]-\tht_{p_r}[\Pa])\right).
\end{equation}
We would like to show that the vector on the right hand side is nonzero. Otherwise, by~\eqref{eq:tht_leq_dots_leq}, we would have $\tht_{p_s}[\Pa]=\cdots=\tht_{p_r}[\Pa]=\tht_{p_1+n}[\Pa]=\cdots=\tht_{p_{s-1}+n}[\Pa]$ for some $s\in[r]$. Let 
\begin{equation*}%
  S:=\{p\in\Z\mid \tht_p[\Pa]=\tht_{p_s}[\Pa]\}.
\end{equation*}
Thus $S$ is a convex subset of $\Pa$ containing all elements in $\sigma^{s-1}(\pch)=(p_s,\dots, p_r,p_1+n,\dots,p_{s-1}+n)$. It follows that $S$ splits as a disjoint union of tubes, all of which belong to $\HAT\setminus\{\Pa\}$. Because $\sigma^{s-1}(\pch)$ is a circular $\Pa$-chain, there exists a path in the Hasse diagram of $\Pa$ which starts at $p_s$, ends at $p_{s-1}+n$, and passes through all elements of $\sigma^{s-1}(\pch)$. For each vertex $p$ on this path, we see that $p\in S$ since $S$ is convex. Thus all elements of $\sigma^{s-1}(\pch)$ belong to the same proper tube $\t'\in \T$. This contradicts the minimality of $\t$. We have shown that the vector on the right hand side of~\eqref{eq:ratmap_Pa} is nonzero, thus $\sinmapb_\pch(\bth)$ is a well defined element of $\RPtnn^{r-1}$ when $\t=\Pa$.

Assume now that $\t\subsetneq\Pa$. Since $\t$ is convex, we may assume after some cyclic relabeling\footnote{Observe that the maps $\sinmap_{\pch}$ and $\sinmap_{\sigma(\pch)}$ are related by a cyclic shift on $\RP^{r-1}$.} that $p_1,p_2,\dots,p_r\in\t$, in which case we set
\begin{equation}\label{eq:ratmap_t}
\sinmapb_\pch(\bth):=\left((\tht_{p_2}[\t]-\tht_{p_1}[\t]):\cdots:(\tht_{p_r}[\t]-\tht_{p_{r-1}}[\t]) :(\tht_{p_r}[\t]-\tht_{p_1}[\t]) \right).
\end{equation}
The entries on the right hand side are nonnegative by~\eqref{eq:th_leq_dots_leq}. Similarly to the above, we see that they cannot all be zero because that would imply $\tht_{p_1}[\t]=\tht_{p_2}[\t]=\cdots=\tht_{p_r}[\t]$, contradicting the minimality of $\t$.

It remains to show that $\sinmapb_\pch$ is continuous. Let $\bth\pari$ be a sequence of elements of $\Cyc(\Pa)$ converging to $\bth$ as $m\to\infty$. By definition, this means that $\bth\pari[\t']$ converges to $\bth[\t']$ inside $\Ord(\t')$ for each non-singleton tube $\t'$. Without loss of generality, we may assume that all points $\bth\pari$ belong to $\D_{\T'}$ for some fixed $\T'\subseteq\T$. Let $\t'\in\HAT'$ be a minimal by inclusion tube containing the residues of $\pch$ modulo $n$. Then $\t\subseteq\t'$. If $\t=\t'$ then clearly $\sinmapb_\pch(\bth\pari)\to\sinmapb_\pch(\bth)$ as $m\to\infty$. If $\t\subsetneq\t'\subsetneq\Pa$ then the result follows from~\eqref{eq:coh}. Finally, if $\t\subsetneq \t'=\Pa$, we see that because $\t\in\T=\T(\bth)$, all coordinates of the vector on the right hand side of~\eqref{eq:ratmap_Pa} tend to zero. But since this vector is treated as an element of $\RP^{r-1}$, we may replace the sines by their arguments. For the last coordinate, we replace $\sin(\tht_{p_1+n}[\Pa]-\tht_{p_r}[\Pa])=\sin(\tht_{p_r}[\Pa]-\tht_{p_1}[\Pa])$ with $\tht_{p_r}[\Pa]-\tht_{p_1}[\Pa]$. Therefore the limit of $\sinmapb_\pch(\bth\pari)$ coincides with the limit of 
\begin{equation}\label{eq:tmp:ratmap1}
  \left((\tht\pari_{p_2}[\Pa]-\tht\pari_{p_1}[\Pa]):\cdots:(\tht\pari_{p_r}[\Pa]-\tht\pari_{p_{r-1}}[\Pa]) :(\tht\pari_{p_r}[\Pa]-\tht\pari_{p_1}[\Pa]) \right)
\end{equation}
as $m\to\infty$. By the coherence condition~\eqref{eq:coh} applied to $\tp:=\Pa$, the vector in~\eqref{eq:tmp:ratmap1} equals
\begin{equation}\label{eq:tmp:ratmap2}
  \left((\tht\pari_{p_2}[\t]-\tht\pari_{p_1}[\t]):\cdots:(\tht\pari_{p_r}[\t]-\tht\pari_{p_{r-1}}[\t]) :(\tht\pari_{p_r}[\t]-\tht\pari_{p_1}[\t]) \right).
\end{equation}
Since $\bth\pari[\t]\to \bth[\t]$ as $m\to\infty$, the vector in~\eqref{eq:tmp:ratmap2} converges to $\sinmapb_\pch(\bth)$.
\end{proof}

\subsection{From bounded affine permutations to affine posets}\label{sec:bap_to_poset}
Suppose that $f\in\Bkn$ is loopless and has a connected strand diagram. Let $\Paf$ be the $n$-periodic transitive closure of the relations $p\lpaf q\lpaf p+n$ whenever $1\leq p<q\leq n$ form an $f$-crossing. (Explicitly, $\lpaf$ is the transitive closure of the relations $p+dn\lpaf q+dn\lpaf p+(d+1)n$ for all $d\in\Z$.) It follows that $\Paf$ is an affine poset. See \cref{fig:strand} for an example.

\begin{figure}
\begin{tabular}{ccc}
\includegraphics[width=0.25\textwidth]{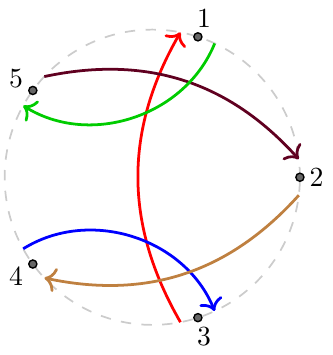}  & &
\includegraphics[width=0.1\textwidth]{fig_tnn/poset_circ}
 \\ \\
Strand diagram of $f\in S_n$ && Affine poset $\Paf$
\end{tabular}
  \caption{\label{fig:strand} Associating an affine poset $\Paf$ (right) to a strand diagram of a permutation $f\in S_n$ for $n=5$ (left). Figure reproduced from~\cite{crit_polyt}.}
\end{figure}

Comparing~\eqref{eq:Oao_Pa} to~\eqref{eq:THtp_dfn}, we see that the sets
\begin{equation*}%
  \Oao(\Paf)=\THtp_f
\end{equation*}
coincide as subsets of $\R^n$. As explained in \cref{rmk:interior_identify}, these spaces are identified with the interior $\Cyco(\Paf)$ of the corresponding affine poset cyclohedron.

\section{Taking the closure}
Suppose that $f\in\Bkn$ is loopless and has a connected strand diagram. Recall from \cref{sec:bap_to_poset} that $\THtp_f$ is naturally identified with the interior $\Cyco(\Paf)$. Thus we have a map
\begin{equation*}%
  \Measf: \Cyco(\Paf)\to \Ctp_f.
\end{equation*}
 Our goal is to show the following result.
\begin{theorem}\label{thm:Measf:Cyc_to_Ctnn}
  For any loopless $f\in\Bkn$, the map $\Measf$ extends to a surjective continuous map between the closures
\begin{equation}\label{eq:Measf:Cyc_to_Ctnn}
  \Measbf: \Cyc(\Paf)\to \Ctnn_f.
\end{equation}
\end{theorem}

First, we describe a simple way to take a limit of a family of boundary measurements. See~\cite[Lemma~3.1]{PSW} for a closely related result.
\begin{lemma}\label{lemma:Gr_limit}
Let $G\in\Gred(f)$. Suppose that we are given a sequence $\wt\pari \in \R_{>0}^{E(G)}$, $m=1,2,\dots$, such that for each $e\in E(G)$, there exists a finite limit
\begin{equation*}%
  \wt(e):=\lim_{m\to\infty} \wt\pari(e) \quad  \in[0,\infty).
\end{equation*}
Let $G'$ be given by
\begin{equation*}%
  V(G'):=V(G),\quad E(G'):=\{e\in E(G)\mid \wt(e)>0\},
\end{equation*}
 and let $\wt'\in\Rtp^{E(G')}$ be the restriction of $\wt$ to $E(G')$. Then we have 
\begin{equation}\label{eq:Meas_converges_to_Meas'}
  \lim_{m\to\infty} \Meas(G,\wt\pari) = \Meas(G',\wt') \quad\text{inside $\Grtnn(k,n)$},
\end{equation}
provided that $G'$ admits at least one almost perfect matching.
\end{lemma}

\begin{proof}
Clearly, we have 
\begin{equation}\label{eq:convergence_R}
  \lim_{m\to\infty} (\Delta_I(G,\wt\pari))_{I\in{[n]\choose k}}  = (\Delta_I(G',\wt'))_{I\in{[n]\choose k}} \quad\text{inside $\R^{n\choose k}$.}
\end{equation}
By construction, any almost perfect matching of $G'$ is an almost perfect matching of $G$. Since the set of such almost perfect matchings is nonempty, the right hand side of~\eqref{eq:convergence_R} is nonzero. Thus~\eqref{eq:convergence_R} also holds inside $\RP^{{n\choose k}-1}$. This implies~\eqref{eq:Meas_converges_to_Meas'}.
\end{proof}

\begin{remark}
We caution that if $G'$ admits no almost perfect matchings, the limit on the left hand side of~\eqref{eq:Meas_converges_to_Meas'} may still exist, since applying a gauge transformation to each $\wt\pari$ may give rise to a different graph $G'$ in the limit.
\end{remark}

Our next goal is to define the map $\Measbf$ in~\eqref{eq:Measf:Cyc_to_Ctnn}. We identify $\Cyc(\Paf)$ with $\Comp(\Paf)$ via \cref{thm:cyc_comp}. Fix $\bth\in\Cyc(\Paf)$ and let $\T:=\T(\bth)$ be the corresponding proper tubing. Choose a contracted graph $G\in\Gred(f)$. 
\begin{lemma}\label{lemma:vertex_chain}
Let $v\in V(G)$ be an interior vertex of $G$ of degree $r$, and let $1\leq p_1<p_2<\dots<p_r\leq n$ be the endpoints of the strands emanating from $v$. Then $(p_1,p_2,\dots, p_r)$ is a circular $\Paf$-chain.
\end{lemma}
\begin{proof}
It is easy to see from the ``no bad double crossings'' condition on the strands~\cite[Theorem~13.2]{Pos} that the edges incident to $v$ are labeled by $\{p_1,p_2\},\dots,\{p_{r-1},p_r\},\{p_r,p_1\}$ in clockwise order. %
The result follows by~\cite[Proposition~4.2]{crit}.
\end{proof}
In the setting of the above lemma, we denote $\Neighv:=(p_1,p_2,\dots,p_r)$. Observe that the entries of $\sinmapb_{\Neighv}(\bth)$ are naturally labeled by $\{p_1,p_2\},\dots,\{p_{r-1},p_r\},\{p_r,p_1\}$; see~\eqref{eq:sin_map}. Thus we may treat the entries of $\sinmapb_{\Neighv}(\bth)$ as nonnegative real edge weights assigned to the edges incident to $v$. They form an element of $\RPtnn^{r-1}$ since rescaling them by a common positive scalar corresponds to a gauge transformation at $v$.
\begin{definition}\label{dfn:Measbf}
Let $\bth\in\Comp(\Paf)$. We define a weight function $\wt_\bth\in\Rtnn^{E(G)}$ as follows. For each boundary edge $e$, set $\wt_\bth(e):=1$. For each black interior vertex $b\in V(G)$, set the weights of the edges incident to $b$ to be proportional to the entries of $\sinmapb_{\Neighb}(\bth)$. Let $G'$ be given by
\begin{equation}\label{eq:Measbf_G'_dfn}
  V(G'):=V(G),\quad E(G'):=\{e\in E(G)\mid \wt_\bth(e)>0\},
\end{equation}
and let $\wt'_\bth$ be the restriction of $\wt_\bth$ to $E(G')$. Define
\begin{equation}\label{eq:Measbf_dfn}
  \Measbf(\bth):=\Meas(G',\wt'_\bth).
\end{equation}
\end{definition}
\noindent See \cref{fig:limit1,fig:limit2} for examples of weighted graphs $(G',\wt')$.

\begin{remark}\label{rmk:Measbf_extends}
For $\bth\in\THtp_f\cong \Oao(\Paf)$, we have $\Measbf(\bth)=\Measf(\bth)$ in view of~\eqref{eq:wt_bth} and \cref{lemma:sin_map}.
\end{remark}
\begin{remark}\label{rmk:independence}
The construction of $\Measbf$ in \cref{dfn:Measbf} formally depends on the choice of $G\in\Gred(f)$. However, we will see later that the choice of $G$ is immaterial: we will show that $\Measbf$ is a \emph{continuous} extension of $\Measf$ to $\Cyc(\Paf)$. If such a continuous extension exists, it must be unique, and thus any other choice of $G$ would give rise to the same map $\Measbf$.
\end{remark}

While the graph $G$ in \cref{dfn:Measbf} was assumed to be reduced and contracted, these properties need not hold for $G'$. But first, in order for~\eqref{eq:Measbf_dfn} to give a well-defined element of the Grassmannian, we must show that not all coordinates of the vector $\Meas(G',\wt'_\bth)$ are zero, which is equivalent to the following statement.
\begin{proposition}\label{prop:at_least_one_apm}
The graph $G'$ given by~\eqref{eq:Measbf_G'_dfn} admits at least one almost perfect matching.
\end{proposition}
\begin{proof}
Recall that we have set $\T:=\T(\bth)$. Our first goal is to show that there exists a maximal proper tube $\t'$ such that $\T\cup\{\t'\}$ is a tubing.

The tubing $\T$ corresponds to a face $\DfT$ of $\Cyc(\Paf)$. Let $\Df_{\T'}$ be any vertex of the closed face $\Closure(\DfT)$. Thus $\T\subseteq\T'$ and $\Df_{\T'}$ is a zero-dimensional face, which means $|\EQ[\T']|=n-1$. We claim that any proper tubing $\T'$ satisfying $|\EQ[\T']|=n-1$ contains a maximal proper tube.

 To see this, consider a rooted tree $T'$ (cf.~\cite[Definition~3.5]{crit_polyt}) with vertex set $\{\Paf\}\sqcup\EQ[\T']\sqcup \Z/n\Z$, where $\Z/n\Z$ is identified with the set of equivalence classes of singleton tubes. (We identify the set of singleton tubes with $\Z$.) The root of $T'$ is $\Paf$, while $\Z/n\Z$ is the set of leaves of $T'$. The children of each $\eq[\t]\in\{\Paf\}\sqcup \EQ[\T']$ are of the form $\eq[\tm]$ where $\tm$ is a maximal by inclusion element of $\T'\sqcup \Z$ satisfying $\tm\subsetneq \t$. We find that $T'$ has $2n$ vertices, including $n$ leaves. Moreover, each non-leaf vertex of $T'$ other than $\Paf$ has at least two children. Since a binary tree on $n$ leaves contains $2n-1$ vertices, it follows that the root $\Paf$ has exactly one child in $T'$. In other words, $\T'$ contains a maximal proper tube $\t'$. Thus $\T\cup\{\t'\}$ is contained in a tubing $\T'$, and therefore is itself a tubing.

We now construct an almost perfect matching $\match$ of $G$. Let $v$ be a (black or white) interior vertex of $G$. Since $\t'$ is a maximal proper tube, it contains the residues of $\Neighv$ modulo $n$, and we let $p_1,p_2,\dots,p_r\in\t'$ be such that $(p_1,p_2,\dots,p_r)=\sigma^s(\Neighv)$ for some $s\in\Z$. Thus the strands emanating from $v$ are labeled by $p_1,p_2,\dots,p_r$ in clockwise order, where we consider their labels modulo $n$.

 We see that $v$ is incident to an edge $e_v$ labeled by $\{p_1,p_r\}$. Set
\begin{equation*}%
  \match:=\{e_v\mid \text{$v$ is an interior vertex of $G$}\}.
\end{equation*}
Thus $\match$ is a collection of edges of $G$ covering each interior vertex at least once. 

Let $b$ (resp., $w$) be a black (resp., white) interior vertex of $G$. We claim that 
\begin{equation}\label{eq:match_iff}
  \text{$e_b$ connects $b$ to $w$ \quad $\Longleftrightarrow$ \quad $e_w$ connects $b$ to $w$.}
\end{equation}
Suppose that $e_b$ connects $b$ to $w$. Label the strands emanating from $b$ (resp., from $w$) by $p_1,p_2,\dots,p_r\in\t'$ (resp., $q_1,q_2,\dots,q_s\in\t'$) in clockwise order. Thus $e_b$ is labeled by $\{p_1,p_r\}$ while $e_w$ is labeled by $\{q_1,q_s\}$. Since $w$ is also incident to the edge $e_b$ labeled by $\{p_1,p_r\}$, we see that $p_1,p_r\in \{q_1,q_2,\dots,q_s\}$, and moreover, $p_r$ appears right before $p_1$ in the sequence $(q_1,q_2,\dots,q_s,q_1)$.
 It follows that $q_1=p_1$ and $q_s=p_r$, therefore $e_w=e_b$. The converse direction is handled similarly, except that for a white interior vertex $w$, $e_w$ may be a boundary edge (in which case there is no black interior vertex $b$ satisfying $e_b=e_w$).

It follows from~\eqref{eq:match_iff} that $\match$ is an almost perfect matching of $G$. It remains to show that $\match$ is an almost perfect matching of $G'$. Recall that $V(G')=V(G)$. Let $b$ be a black interior vertex of $G$ with outgoing strands labeled by $p_1,p_2,\dots,p_r\in\t'$. Thus the edge $e_b\in\match$ is labeled by $\{p_1,p_r\}$. Our goal is to show that the $r$-th entry $y_r$ of $\by:=\sinmapb_{(p_1,p_2,\dots,p_r)}(\bth)$ is nonzero. By \cref{lemma:sin_map}, the entries of  $\by$ are not all zero. 

 Let $\t\in\HAT$ be a minimal by inclusion tube containing the residues of $(p_1,p_2,\dots,p_r)$ modulo $n$. We first consider the case $\t=\Paf$. By~\eqref{eq:ratmap_Pa}, we have $y_r=\sin(\tht_{p_1+n}[\t]-\tht_{p_r}[\t]).$ 
Thus by~\eqref{eq:tht_leq_dots_leq}, $y_r=0$ implies that either $\tht_{p_r}[\t]=\tht_{p_1}[\t]$ or $\tht_{p_r}[\t]=\tht_{p_1}[\t]+\pi$. In the former case, the vector $\by$ would be zero, a contradiction. Thus assume $\tht_{p_r}[\t]=\tht_{p_1}[\t]+\pi$. Let 
\begin{equation*}%
  S:=\{p\in \Z\mid \tht_p[\t]=\tht_{p_r}[\t]\}.
\end{equation*}
We see that $S$ is a convex subset of $\Paf$ containing both $p_r$ and $p_1+n$. Recall that $S$ is a disjoint union of tubes. Since $p_r\lpaf p_1+n$, these two elements belong to the same connected component $\tm$ of $S$. By \cref{dfn:T(bth)}, we must have $\tm\in\T$. This is a contradiction: $\T\cup\{\t'\}$ is a tubing, however, the tubes $\tm,\t'\in\T\cup\{\t'\}$ are neither nested nor disjoint. We have shown that $y_r\neq0$ when $\t=\Paf$.

Assume now that $\t\subsetneq\Paf$ is a proper tube. By choosing a particular representative in $\eq[\t]$, we may assume that $p_1\in\t$. Since any two tubes in $\T\cup\{\t'\}$ are either nested or disjoint, and since $p_1\in\t\cap\t'$, we must have $\t\subseteq\t'$. (Because $|\t'|=n$, we cannot have $\t'\subsetneq \t$.) It follows that $p_1,p_2,\dots,p_r\in\t$. Since $\t\subsetneq\Paf$, $\by$ is given by~\eqref{eq:ratmap_t}. In particular, $y_r=\tht_{p_r}[\t]-\tht_{p_1}[\t]$. By~\eqref{eq:tht_leq_dots_leq}, $y_r=0$ implies $\by=0$, a contradiction.
\end{proof}

\begin{proof}[Proof of \cref{thm:Measf:Cyc_to_Ctnn}]
By \cref{prop:at_least_one_apm}, the map $\Measbf$ lands inside $\Gr(k,n)$. By \cref{rmk:Measbf_extends}, it extends the map $\Measf$ to $\Cyc(\Paf)$. Next, we show that it is continuous.

Let $(\bth\pari)_{m\geq1}$ be a sequence of points in $\Cyc(\Paf)$ converging to $\bth\in\Cyc(\Paf)$ as $m\to\infty$. Let $b$ be a black interior vertex of $G$ of degree $r$. By \cref{lemma:sin_map}, the map $\sinmapb_{\Neighb}$ is continuous on $\Cyc(\Paf)$:
\begin{equation*}%
  \lim_{m\to\infty} \sinmapb_{\Neighb}(\bth\pari)=\sinmapb_{\Neighb}(\bth) \quad\text{inside $\RP^{r-1}$}.
\end{equation*}
Thus, after applying gauge transformations to each $\wt_{\bth\pari}$ at black interior vertices, we get
\begin{equation*}%
  \lim_{m\to\infty} \wt_{\bth\pari}(e)=\wt_{\bth}(e) \quad\text{for all $e\in E(G)$}.
\end{equation*}
(Recall that the weight of each boundary edge $e$ is not affected by gauge transformations at black interior vertices, and satisfies $\wt_{\bth\pari}(e)=\wt_{\bth}(e)=1$ for all $m$.) By \cref{prop:at_least_one_apm}, $G'$ admits an almost perfect matching, therefore $\Measbf$ is continuous by \cref{lemma:Gr_limit}.

It remains to show that $\Measbf(\Cyc(\Paf))=\Ctnn_f$. We see that the image of $\Measbf$ is compact (since $\Cyc(\Paf)$ is compact) and thus closed. Since the image contains $\Ctp_f=\Measf(\Cyco(\Paf))$, it contains the closure $\Ctnn_f$ of $\Ctp_f$. On the other hand, $\Measbf(\Cyc(\Paf))$ must be contained inside $\Ctnn_f$ because $\Cyc(\Paf)$ is the closure of $\Cyco(\Paf)$.
\end{proof}

\begin{definition}\label{dfn:strat}
We endow $\Ctnn_f$ with a stratification obtained by taking the common refinement of the images of  all open faces of $\Cyc(\Paf)$. 
\end{definition}

\begin{conjecture}
For any two open faces $\DfT,\Df_{\T'}$ of $\Cyc(\Paf)$, their images under $\Measbf$ either coincide or are disjoint.
\end{conjecture}
\noindent Below we prove this conjecture for $f=\fkn$. %

\section{Top cell and the second hypersimplex}
We concentrate on the case of the top cell ($f=\fkn$), where $2\leq k\leq n-1$. We denote $\Ctnnkn:=\Ctnn_{\fkn}$, $\Pakn:=\Pa_{\fkn}$, etc. Note that $\Cyc(\Pakn)\cong\Cyc_n$ is just the standard $(n-1)$-dimensional cyclohedron of~\cite{BoTa,Simion}. Our goal is to prove \cref{thm:intro:hyp2n}.

\subsection{\texorpdfstring{From $\Cyc(\Pakn)$ to $\Delta_{2,n}$}{From the cyclohedron to the second hypersimplex}}
Recall from \cref{thm:Measf:Cyc_to_Ctnn} that $\Ctnn_{k,n}$ is the image of the cyclohedron $\Cyc(\Pakn)$ under the map $\Measbkn:\Cyc(\Pakn)\to\Ctnnkn$. Our first goal is to introduce a map $\mapCD:\Cyc(\Pakn)\to\Delta_{2,n}$ to the second hypersimplex and to show that $\Measbkn$ factors through $\mapCD$.

We start with a few preliminary observations and definitions. 

\begin{notation}\label{notn:cyclic}
 For $a,b\in\Z$ with $a\leq b$, we set $[a,b):=\{a,a+1,\dots,b-1\}$. For $a,b\in[n]$, we introduce a \emph{cyclic interval} $[a,b):=\{a,a+1,\dots,b-1\}$ if $a\leq b$ and $[a,b):=\{a,a+1,\dots,n,1,\dots,b-1\}$ if $a>b$. The intervals $(a,b],[a,b]\subseteq\Z$ (for $a\leq b$) and cyclic intervals $(a,b],[a,b]\subseteq[n]$ (for $a,b\in[n]$) are defined analogously.
\end{notation}
\begin{definition}
An \emph{inscribed polygon} (resp., \emph{degenerate inscribed polygon}) is a polygon all of whose vertices lie on a single circle (resp., on a single line).
\end{definition}
We view (degenerate) inscribed polygons modulo transformations that preserve the ratios of the distances between their vertices. We write $R=(v_1,v_2,\dots,v_m)$ for a polygon with vertices $v_1,v_2,\dots,v_m$ given in cyclic order. The following result is well known.
\begin{lemma}\label{lemma:polygon}
Let $(a_1,a_2,\dots,a_m)\in\Rtnn^m$ be such that $a_p\leq \sum_{q\neq p} a_q$ for all $p\in[m]$. Then there exists a unique possibly degenerate inscribed polygon $R=(v_1,v_2,\dots,v_m)$ such that 
\begin{equation*}%
  \frac{|v_{p+1}-v_p|}{|v_{q+1}-v_q|}=\frac{a_p}{a_q} \quad\text{for all $p,q\in[m]$,}
\end{equation*}
where we set $v_{m+1}:=v_1$.\qed
\end{lemma}
\noindent Thus, up to a common scalar, the diagonals of a possibly degenerate inscribed polygon may be reconstructed from its sides.

Next, observe that the order $\leqpakn$ coincides with the usual total order $\leq$ on $\Z$. In particular, $\pchn:=(1,2,\dots,n)$ is a circular $\Pakn$-chain. By \cref{lemma:sin_map}, we therefore have a continuous map
\begin{equation*}%
  \sinmapb_\pchn:\Cyc(\Pakn)\to \RPtnn^{n-1}.
\end{equation*}

\begin{definition}
Let $\bth\in\Cyc(\Pakn)$. We denote by $\BB_\bth$ the partition of $\Z$ into intervals consisting of the tubes in $\tubes(\bth[\t])$ for each minimal by inclusion $\t\in\HAT(\bth)$ satisfying $|\t|=n$. 
\end{definition}
In the above definition, either $\t=\Pakn$ or $\t$ is a maximal proper tube, which in the case of $\Pakn$ is just an interval of the form $[p,p+n)\subseteq\Z$ for some $p\in\Z$. Thus $\tubes(\bth[\t])$ forms a partition of $\t$ into intervals. Considering $\BB_\bth$ modulo $n$, we get a partition $\BBmod_\bth=(\Bmod_1,\Bmod_2,\dots,\Bmod_m)$ of $[n]$ into $m\geq2$ nonempty cyclic intervals.

\begin{remark}\label{rmk:polygons_k_n}
Recall from \cref{rmk:exp_simplex} that for $\bth\in \Cyco(\Pakn)$,  setting $v_r:=\exp(2i\th_r)$ for $r\in[n]$ gives $n$ distinct points $v_1,v_2,\dots,v_n$ on the unit circle ordered counterclockwise. The map $\sinmap_\pchn$ in this case records the side length ratios of the $n$-gon $R=(v_1,v_2,\dots,v_n)$. When we pass to the boundary ($\bth\in\Cyc(\Pakn)$), some of these points will collide. If not all points collide then the cyclic intervals in $\BBmod_\bth=(\Bmod_1,\Bmod_2,\dots,\Bmod_m)$ record precisely the groups of collided points, and $\sinmapb_\pchn(\bth)$ records the side length ratios of the corresponding $m$-gon. If all points collide then $\T(\bth)$ contains a maximal proper tube $\t$. In this case, $\bth[\t]$ records the positions of $n$ points on a line, the cyclic intervals in $\BBmod_\bth$ record which groups of those points collided together, and $\sinmapb_{\pch}(\bth)$ records the side length ratios of the corresponding degenerate inscribed $m$-gon.
\end{remark}

Consider a map 
\begin{equation*}%
  \divsum: \RPtnn^{n-1}\to \Rtnn^n,\quad (x_1:x_2:\cdots:x_n)\mapsto \frac2{x_1+x_2+\cdots+x_n}(x_1,x_2,\dots,x_n).
\end{equation*}
We note that the entries of an element of $\RPtnn^{n-1}$ are nonnegative and at least one of them is nonzero, thus their sum is strictly positive. The image of $\divsum$ belongs to the subspace of $\Rtnn^n$ where the sum of coordinates is equal to $2$. Let
\begin{equation*}%
  \mapCD: \Cyc(\Pakn)\to \Rtnn^n,\quad \mapCD:=\divsum\circ \sinmapb_{\pchn}.
\end{equation*}
\begin{proposition}\label{prop:delta2n}
The image of the map $\mapCD$ equals 
\begin{equation*}%
  \Delta_{2,n}=\{(y_1,y_2,\dots,y_n)\in[0,1]^n\mid y_1+y_2+\cdots+y_n=2\}.
\end{equation*}
\end{proposition}
\begin{proof}
By \cref{rmk:polygons_k_n}, the map $\sinmap_{\pchn}$ records the side length ratios of an inscribed $n$-gon, and thus its image is described by triangle inequalities:
\begin{equation*}%
  \sinmap_{\pchn}(\Cyco(\Pakn))=\{(x_1:x_2:\cdots:x_n)\in\RPtnn^{n-1}\mid 0< x_p< \sum_{q\neq p} x_q\text{ for each $p\in[n]$}\}.
\end{equation*}
Observe that $0<x_p< \sum_{q\neq p} x_q$ is equivalent to $0<2x_p< \sum_{q=1}^n x_q$. Substituting $y_p:=\frac{2x_p}{x_1+x_2+\cdots+x_n}$, we get
\begin{equation*}%
  \mapCD(\Cyco(\Pakn))=\{(y_1,y_2,\dots,y_n)\in\R^n\mid 0<y_p<1\text{ for each $p\in[n]$ and }y_1+y_2+\cdots+y_n=2\}.
\end{equation*}
The result follows by taking the closure.
\end{proof}

\subsection{\texorpdfstring{From $\Delta_{2,n}$ to $\Ctnnkn$}{From the second hypersimplex to the totally nonnegative critical variety}}\label{sec:DC}
The goal of this section is to prove the following result.
\begin{theorem}\label{thm:map_factors} 
There exists a continuous map
\begin{equation*}%
  \mapDC:\Delta_{2,n}\to\Ctnnkn
\end{equation*}
making the diagram~\eqref{eq:intro:map_factors} commutative.
\end{theorem}
Thus, \cref{thm:intro:hyp2n} consists of \cref{thm:map_factors} together with the statement that the map $\mapDC$ is a homeomorphism, which we prove in \cref{sec:inj}.

Let $\bth\in\Cyc(\Pakn)$. Since $\Measbkn(\bth)\in\Grtnn(k,n)$, it must belong to some positroid cell $\Povtp_{g}$, $g\in\Bkn$. We will see later (\cref{prop:graph_reduction}) that the bounded affine permutation $g$ has the following description. For a subset $A\subseteq\Z$ and $p\in\Z$, we let $A+p:=\{a+p\mid a\in A\}$. By an \emph{$n$-periodic interval partition of $\Z$} we mean a collection $\BB$ of disjoint nonempty intervals in $\Z$ of size strictly less than $n$ such that their union is $\Z$ and for each interval $\B\in\BB$, we have $\B+dn\in\BB$ for all $d\in\Z$.
\begin{lemma}\label{lemma:gBB}
For any $n$-periodic interval partition $\BB$ of $\Z$, there exists a unique loopless $\gBB\in\Bkn$ of maximal length such that $\gBB(\B-k)=\B$ for all $\B\in\BB$.
\end{lemma} 
\begin{figure}
\includegraphics[width=1.0\textwidth]{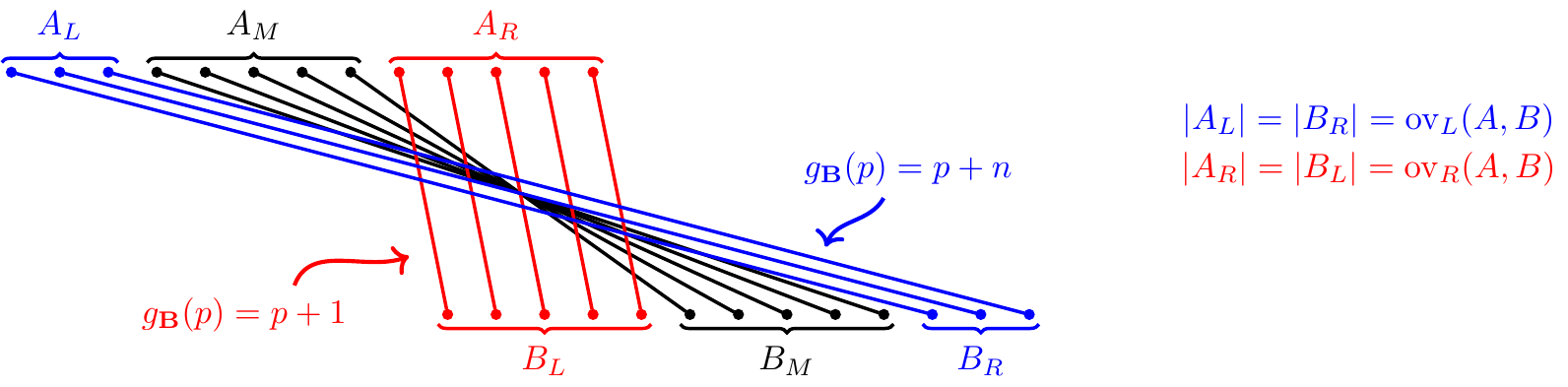} 
  \caption{\label{fig:gb_perm} The permutation $\gBB$ described in \cref{lemma:gBB}.}
\end{figure}
\begin{proof}
We describe $\gBB$ explicitly; see \cref{fig:gb_perm}. Let $\B\in\BB$ and denote $\A:=\B-k$. Let
\begin{equation}\label{eq:ovl_ovr}
  \ovl(\A,\B):=|(\A+n)\cap \B| \quad\text{and}\quad \ovr(\A,\B):=|(\A+1)\cap \B|.
\end{equation}
We have $\ovl(\A,\B)+\ovr(\A,\B)\leq|\A|=|\B|<n$. Let $\A_L$ consist of the smallest $\ovl(\A,\B)$ elements of $\A$, let $\A_R$ consist of the largest $\ovr(\A,\B)$ elements of $\A$, and let $\A_M$ consist of the remaining elements of $\A$. Thus we have a partition $\A=\A_L\sqcup \A_M\sqcup \A_R$ into intervals. Next, we partition $\B=\B_L\sqcup \B_M\sqcup \B_R$ into intervals given by $\B_L=\A_R+1$ and $\B_R:=\A_L+n$. For $p\in \A_R$, we let $\gBB(p):=p+1\in\B_L$, and for $p\in\A_L$, we let $\gBB(p):=p+n\in\B_R$. The restriction of $\gBB$ to $\A_M$ is an order reversing bijection $\A_M\to\B_M$. This ensures that $\gBB$ has maximal possible length among all loopless bounded affine permutations sending $A$ to $B$. It is also clear that $\gBB\in\Bkn$ (as opposed to $\Bound(k',n)$ for some $k'\neq k$) since it can be obtained from $\fkn$ by applying (affine) simple transpositions.
\end{proof}

Recall from \cref{rmk:independence} that any choice of a graph $G\in\Gred(\fkn)$ gives rise to the same map $\Measbkn$. We will take advantage of this observation by using a particular graph $\Gkn\in\Gred(\fkn)$ called the \emph{Le-diagram graph}; see \figref{fig:Le}(a) for an example and~\cite[Section~20]{Pos} for background. 

\begin{notation}\label{notn:strands}
All interior vertices of $\Gkn$ have degree either $2$ or $3$. Each interior vertex $v$ belongs to one horizontal strand directed east, one vertical strand directed south, and one diagonal strand directed northwest; see \figref{fig:Le}(b). We denote the endpoints of these strands by $\E(v),\S(v),\NW(v)\in[n]$, respectively. If a black vertex $b$ has degree $2$ then we have $\NW(b)=\S(b)$. If a white vertex $w$ has degree $2$ then we have $\NW(w)=\E(w)$. We denote by $\VB$ the set of black interior vertices of $\Gkn$.%
\end{notation}
\noindent Thus each $b\in\VB$ is uniquely determined by $\E(b)$ and $\S(b)$, which are its vertical and horizontal coordinates in the plane.

\begin{figure}
\begin{tabular}{cc}
\includegraphics[width=0.45\textwidth]{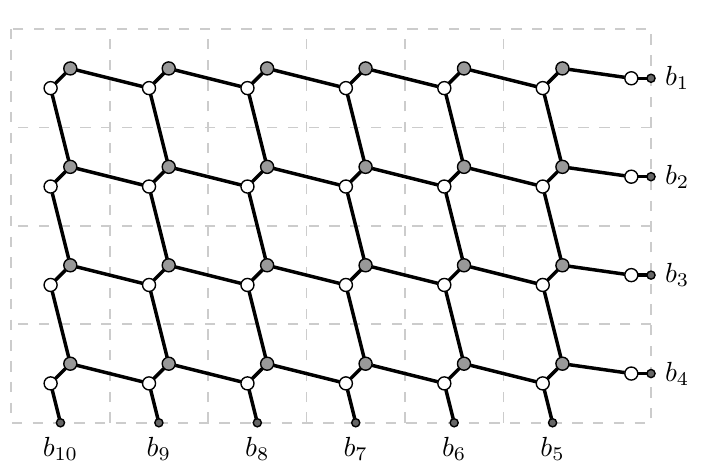} 
&
\includegraphics[width=0.45\textwidth]{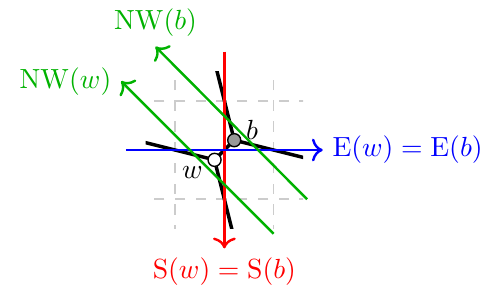} 
\\
(a) Le-diagram graph $\Gkn$ for $k=4$, $n=10$ & (b) Strands in $\Gkn$; see \cref{notn:strands}
\end{tabular}
  \caption{\label{fig:Le} A Le-diagram graph and its strands.}
\end{figure}

 After a cyclic shift, we may assume that
\begin{equation}\label{eq:BB_different}
  \text{$k$ and $k+1$ belong to different intervals in $\BB_\bth$.} 
\end{equation}
\begin{definition}\label{dfn:special}
An interval $\B\in\BB_\bth$ is called \emph{special} if it contains both $n$ and $n+1$. 
We also refer to the corresponding cyclic interval $\Bmod\in\BBmod_\bth$ as \emph{special}.
\end{definition}
\noindent Clearly, $\BB_\bth$ contains at most one special interval. 

Next, we consider the weighted graph $(G',\wt')$ obtained from $\Gkn$ via \cref{dfn:Measbf}. 

\begin{definition}
We say that $b\in \VB$ is of \emph{\type(1)} (resp., \emph{\type(2)} or \emph{\type(3)}) if the endpoints of the strands emanating from $b$ belong to exactly one (resp., two or three) distinct cyclic intervals in $\BBmod_\bth$.
\end{definition}

\begin{remark}\label{rmk:types}
 If $b$ is of \type(3), all three edges of $b$ are present in $G'$. Their weights coincide with their weights in $G$, and can be computed from $\mapCD(\bth)$; cf. \cref{rmk:polygons_k_n} and \cref{lemma:polygon}. If $b$ is of \type(2), only two edges of $b$ are present in $G'$. Their weights are equal, and after a gauge transformation at $b$, can be made equal to $1$. Finally, if $b$ is of \type(1), either two or three edges of $b$ are present in $G'$, and their weights cannot in general be computed from $\mapCD(\bth)$.  See e.g. \cref{fig:limit2,fig:Le_hv_rect}.
\end{remark}
\begin{lemma}\label{lemma:type1}
If $\BBmod_\bth$ does not contain a special  cyclic interval (in the sense of \cref{dfn:special}) then $\VB$ contains no vertices of \type(1). If $\BBmod_\bth$ contains a special cyclic interval $\Bmod$ then for each $b\in\VB$, $b$ is of \type(1) if and only if $\S(b),\E(b)\in\Bmod$.
\end{lemma}
\begin{proof}
In order for $b\in\VB$ to be of \type(1), $\S(b),\E(b),\NW(b)$ must belong to some cyclic interval $\Bmod\in\BBmod_\bth$. But since $\S(b)\in[k+1,n]$ and $\E(b)\in[k]$, $\Bmod$ must be special in view of~\eqref{eq:BB_different}. Conversely, suppose that $\Bmod\in\BB_\bth$ is special and $\S(b),\E(b)\in\Bmod$. Since $\Bmod$ is of the form $[n-h+1,n]\sqcup [v]$ for some $v\in[k]$ and $h\in[n-k]$,  $\S(b),\E(b)\in\Bmod$ implies $\NW(b)\in\Bmod$.
\end{proof}
\noindent Thus the set of \type(1) vertices forms a top left justified $h\times v$ rectangle in $\Gkn$; see \cref{fig:Le_hv_rect} for an example.

\begin{figure}
  \resizebox{1.0\textwidth}{!}{%
\begin{tabular}{cc}
\includegraphics[width=0.5\textwidth]{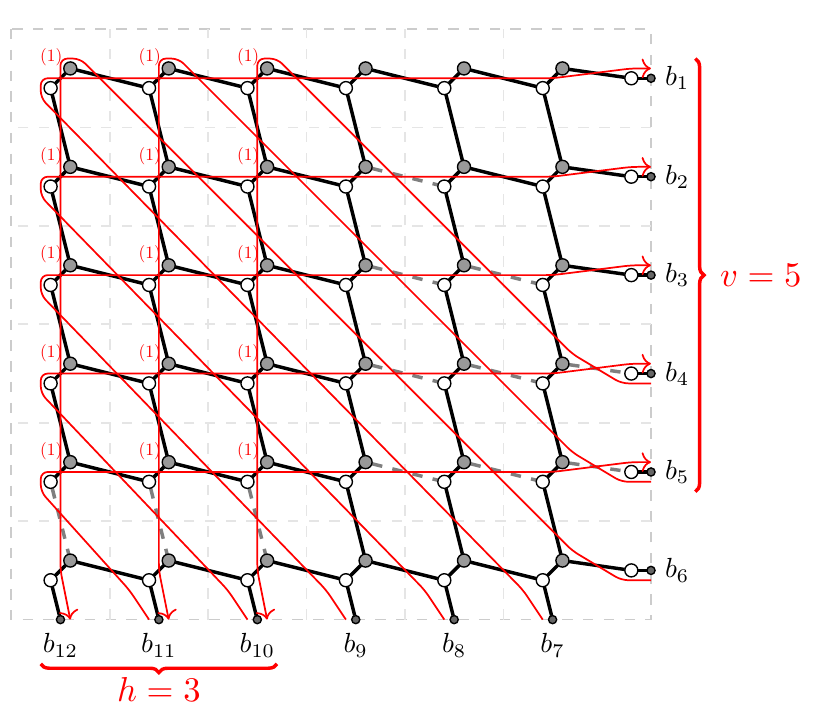} & 
\includegraphics[width=0.5\textwidth]{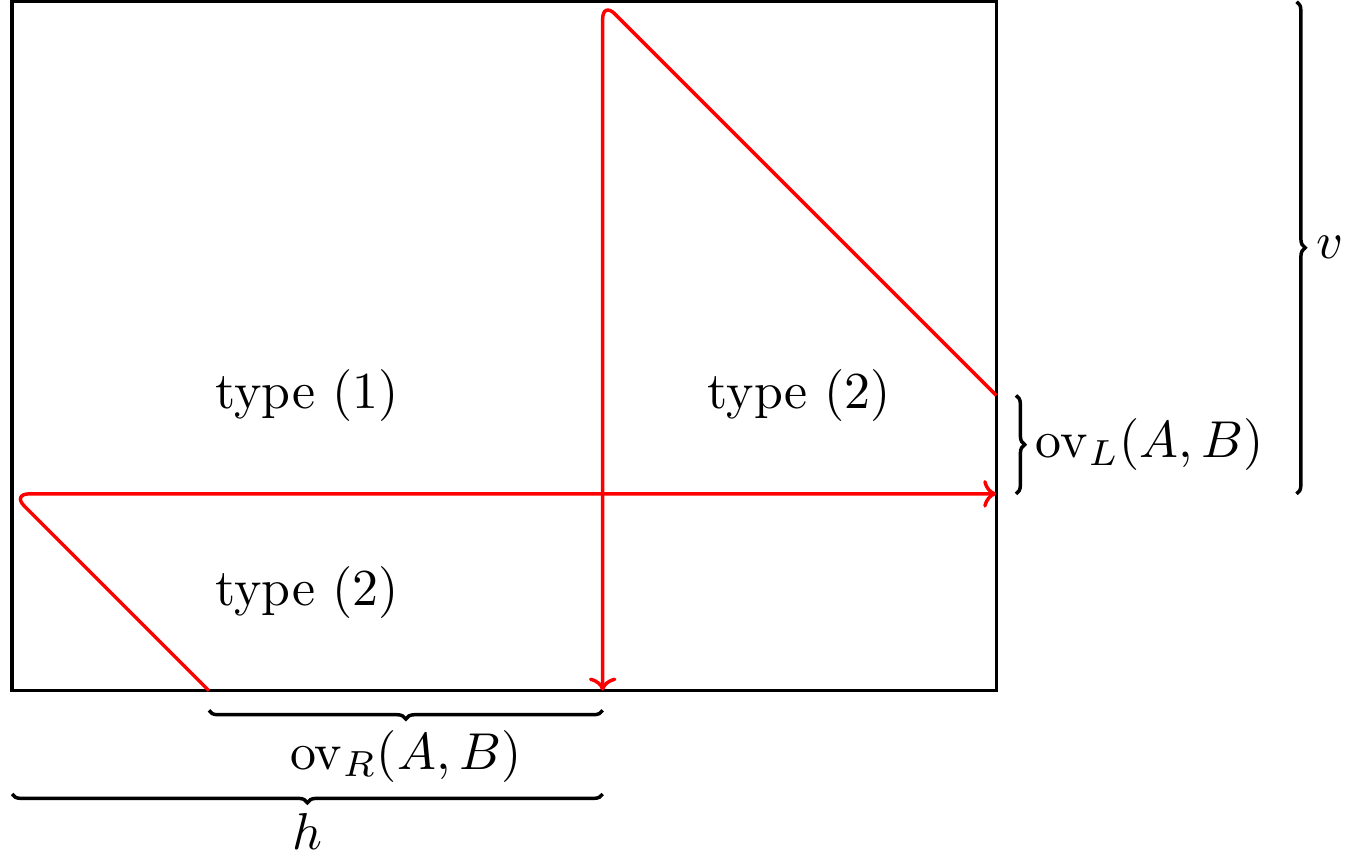} 
\end{tabular}
}
  \caption{\label{fig:Le_hv_rect} Left: an example for \cref{lemma:type1}. Here $k=6$, $n=12$, and $\BBmod_\bth$ contains a special cyclic interval $\Bmod=[10,5]=[10,12]\sqcup[1,5]$. The strands terminating in $\Bmod$ are shown in red. Type (1) black vertices are marked by (1). Dashed edges are present in $G$ but not in $G'$. Each of them is incident to a black vertex of \type(2). Right: regions of $\Gkn$ containing vertices of types (1) and (2) with respect to the special region $\Bmod$.}
\end{figure}

For the next result, we need to refer explicitly to the edges of $\Gkn$. Each black vertex $b\in\VB$ of degree $3$ is incident to a \emph{northern}, \emph{eastern}, and \emph{southwestern} edge labeled by $\{\S(b),\NW(b)\}$, $\{\E(b),\NW(b)\}$, and $\{\S(b),\E(b)\}$, respectively; see \figref{fig:Le}(b). Recall from~\eqref{eq:BB_different} that $k$ and $k+1$ cannot both belong to the special interval in $\BBmod_\bth$. We let $\gbth:=g_{\BB_\bth}$ be given by \cref{lemma:gBB}. Let us say that a \emph{self-loop} is an edge of a graph connecting a vertex to itself.

\begin{proposition}\label{prop:graph_reduction}
If $\BBmod_\bth$ does not contain a special cyclic interval then set $G'':=G'$. Otherwise, let $\Bmod$ be the special cyclic interval of $\BBmod_\bth$, and let $G''$ be obtained from $G'$ in one of the following two ways:
\begin{itemize}
\item (if $k\notin\Bmod$) remove all black vertices of \type(1) and their southwestern white neighbors; %
\item (if $k+1\notin\Bmod$) contract all edges incident to black vertices of \types(1,2) and remove all self-loops in the resulting graph.
\end{itemize}
 Let $\wt''_\bth$ be the restriction of $\wt'_\bth$ to the edges of $G''$. Then
\begin{equation}\label{eq:G''_red}
  G''\in\Gred(\gbth) \quad\text{and}\quad \Meas(G'',\wt''_\bth)=\Meas(G',\wt'_\bth).
\end{equation}
\end{proposition}
\noindent For the example in \figref{fig:Le_hv_rect}(left), we have $k,k+1\notin \Bmod$, so either of the two above procedures yields a reduced graph $G''$ satisfying the conditions in~\eqref{eq:G''_red}.
\begin{proof}
 Consider a vertex $b\in\VB$ of \type(1) and let $e$ be its southwestern edge. We claim that $\wt_\bth(e)$ equals the sum of $\wt_\bth(e')$ over all other edges $e'$ of $b$. (In particular, $\wt_\bth(e)>0$ so $e$ is present in $G'$.) Indeed, this is clear if $b$ has degree $2$. If $b$ has degree $3$ then $\NW(b)$ belongs to the cyclic interval $[\S(b),\E(b)]$. Thus there exists a $\Pakn$-circular chain $(p,q,r)$ such that $p,q,r$ are equal respectively to $\S(b)$, $\NW(b)$, $\E(b)$ modulo $n$, and such that $p,q,r\in\t$ for some proper tube $\t\in\T(\bth)$. This implies that $\wt_\bth(e)=\wt_\bth(e')+\wt_\bth(e'')$, where $e,e',e''$ are labeled by $\{p,r\}$, $\{p,q\}$, and $\{q,r\}$, respectively. 

Let $\Bmod:=[n-h+1,n]\sqcup[v]$ for $v\in[k]$ and $h\in[n-k]$; see \cref{fig:Le_hv_rect}. Assume first that $k\notin\Bmod$, thus $v<k$. Let $b\in\VB$ be a vertex satisfying $\E(b)=v+1$ and $\S(b)\in\Bmod$. Then $b$ is of \type(2) with $\S(b),\NW(b)\in\Bmod$, and its northern edge labeled by $\{\S(b),\NW(b)\}$ is not present in $G'$. Thus the bottom left black vertex $b$ of \type(1) (defined by $\S(b)=n$, $\E(b)=v$) is adjacent to a white vertex of degree $1$ in $G'$. Applying a sequence of leaf removals (\figref{fig:moves_red}(middle)) starting with $b$ and proceeding up and to the right, we remove all black vertices of \type(1) and their southwestern white neighbors.

Assume now that $k+1\notin\Bmod$, thus $h<n-k$. Let $b'\in\VB$ be a vertex satisfying $\S(b')=n-h$ and $\E(b')\in[2,v]$. Then $b'$ is of \type(2) with $\E(b),\NW(b)\in\Bmod$, and its eastern edge labeled by $\{\E(b),\NW(b)\}$ is not present in $G'$. For a connected subgraph $H$ of $G'$, let $G'/H$ be obtained from $G'$ by contracting all edges in $H$ and removing all self-loops in the resulting graph. Initialize $H$ to consist of all edges incident to vertices $b\in\VB$ of \type(2). This includes the edges incident to black vertices at the top ($\E(b)=1$, $\S(b)\in\Bmod$) and the right ($\S(b)=n-h$, $\E(b)\in\Bmod$) boundaries of the $(h+1)\times v$ rectangle. Choose the top right \type(1) black vertex that is not a vertex of $H$. Its northern and eastern white neighbors are in $H$. Let $e,e',e''$ be the edges of $G$ incident to $b$  as above (where one of $e',e''$ may not be present in $G'$), so that $e$ is the southwestern edge. If both $e',e''$ are present then their images in $G'/H$ form a double edge. Applying a parallel edge reduction move (\figref{fig:moves_red}(left)), we transform this double edge into a single edge of weight $\wt_\bth(e')+\wt_\bth(e'')$, which, as we have shown above, equals $\wt_\bth(e)$. Thus the image of $b$ in $G/H$ is a vertex of degree $2$, and the two edges incident to it have the same weight. These two edges may be contracted using a contraction-uncontraction move (\figref{fig:moves_cs}(left)). This corresponds to adding $e,e',e''$ and their endpoints to $H$, and constitutes the induction step. Once all \type(1) vertices have been added to $H$, we arrive at $G'/H=G''$.

A straightforward consequence of the above construction is that $G''$ has strand permutation $f_{G''}=\gbth$ and satisfies $\Meas(G'',\wt''_\bth)=\Meas(G',\wt'_\bth)$. Indeed, we have $\Meas(G'',\wt''_\bth)=\Meas(G',\wt'_\bth)$ since $(G'',\wt''_\bth)$ was obtained from $(G',\wt'_\bth)$ via a sequence of moves in \cref{fig:moves_red,fig:moves_cs}. To see that $f_{G''}=\gbth$, we first observe directly that $f_{G''}(p)=p+n$ (resp., $f_{G''}(p)=p+1$) if and only if $\gbth(p)=p+n$ (resp., $\gbth(p)=p+1$). Next, since our edge removals taking $G$ to $G'$ and reduction moves taking $G'$ to $G''$ only involved edges labeled by $\{p,q\}$ where $p,q$ belong to a single interval $\Bmod'$ of $\BBmod$, we have $f_{G''}^{-1}(\B')=\B'-k=\gbth^{-1}(\B')$  for any interval $\B'\in\BB$. Finally, it is clear that the contracted version of $G''$ contains no edge labeled $\{p,q\}$ where $p,q$ belong to the same cyclic interval in $\BBmod$. Thus no two strands terminating at any given $\B'$ form a crossing, so $f_{G''}$ coincides with $\gbth$. 

We further note that for any black interior vertex $b\in V(G'')$, 
\begin{equation}\label{eq:weights}
  \text{the weights $\wt_\bth''(e)$ of the edges $e$ of $G''$ incident to $b$ are proportional to $\sinmapb_{\NeighXb_{G''}}(\bth)$.}
\end{equation}
Morally, the last property is close to the statement $\Meas(G'',\wt''_\bth)=\Measbop_{\gbth}(\bth)$, except that we have not yet shown that $G''$ is reduced, and we also have not defined the map $\Measbop_{g}$ for the case when $g$ does not have a connected strand diagram (cf. \cref{dfn:strand}).

In order to complete the proof of the proposition, we need to show that $G''$ is reduced. For that, we will use the following well-known characterization~\cite{Pos} of reduced graphs: $G''$ is reduced if and only if it has no isolated connected components and has exactly $k(n-k)+1-\ell(\gbth)$ faces. It is not hard to check that $G''$ has no isolated connected components.  Since $\Gkn$ has $k(n-k)+1$ faces, we need to show that our process above decreases the number of faces precisely by $\ell(\gbth)$. Since each affine inversion of $\gbth$ involves two strands with endpoints in the same interval of $\BB_\bth$, it suffices to show, for each interval $\B$ of $\BB_\bth$, that the number of affine inversions involving indices from $\B$ matches the number of faces removed from $G$ due to deleting/contracting edges labeled by $\{p,q\}$ for $p,q\in\Bmod$.

Let $\B\in\BB_\bth$, and let $\A:=\B-k$. It follows from the proof of \cref{lemma:gBB} that the number of affine inversions of the restriction of $\gbth$ to $\A$ equals
\begin{equation}\label{eq:BB_aff_inv}
  {|\B|\choose 2}-{\AL\choose2}-{\COL\choose2}.
\end{equation}
If $\B$ is not special then we see that~\eqref{eq:BB_aff_inv} also describes the number of \type(2) vertices involving two indices in $\Bmod$. Indeed, if $\B$ is not special then either $\Bmod\subseteq[k+1,n]$ or $\Bmod\subseteq[k]$. In the former case, we have $\COL=0$ and the number of \type(2) vertices involving two indices in $\B$ equals ${|\B|\choose 2}-{\AL\choose2}$.
 In the latter case, we have $\AL=0$ and the number of \type(2) vertices involving two indices in $\B$ equals ${|\B|\choose 2}-{\COL\choose2}$. Each such \type(2) vertex is incident to an edge of $G$ which is not present in $G'$. We therefore see that in both cases, the number of faces decreases by the quantity given in~\eqref{eq:BB_aff_inv}.

We concentrate on the case where $\B$ is special, so assume $\Bmod=[n-h+1,n]\sqcup[v]$. Either of the two ways to reduce $G'$ to $G''$ removes exactly $h(v-1)$ faces contained in the rectangular region. (When $k+1\in\Bmod$, this includes joining the $v-1$ boundary faces contained between the boundary vertices $b_p$ for $p\in[v]$ into a single boundary face.)
 Next, we count the number of edges removed when passing from $G$ to $G'$. All of them are adjacent to \type(2) black vertices, and are contained in two trapezoidal regions shown in \figref{fig:Le_hv_rect}(right).
 The lower left (resp., upper right) region is a trapezoid if $\AL>0$ (resp., $\COL>0$) and a triangle if $\AL=0$ (resp., $\COL=0$). It contains ${h+1\choose2}-{\AL\choose 2}$ (resp., ${v\choose2}-{\COL\choose2}$)  vertices of \type(2) involving two indices in $\B$. The result follows since 
\begin{equation*}%
  h(v-1)+{h+1\choose2}+{v\choose2}={h+v\choose2}={|\B|\choose 2}.\qedhere
\end{equation*}
\end{proof}

\begin{proof}[Proof of \cref{thm:map_factors}]
By \cref{dfn:Measbf}, we have $\Measbf(\bth)=\Meas(G',\wt'_\bth)$, which equals $\Meas(G'',\wt''_\bth)$ by \cref{prop:graph_reduction}. By \cref{rmk:types}, the edge weights of $G''$ may be computed purely in terms of the side length ratios encoded in $\mapCD(\bth)$. Thus $\Measbf$ factors through $\mapCD$. Since $\mapCD$ is surjective, there exists a unique map $\mapDC:\Delta_{2,n}\to\Ctnn_{k,n}$ making the diagram~\eqref{eq:intro:map_factors} commutative. It remains to show that $\mapDC$ is continuous. Letting $X:=\Cyc(\Pakn)$, $Y:=\Delta_{2,n}$, and $Z:=\Ctnn_{k,n}$, we have maps $X\xrightarrow{\mapCD} Y \xrightarrow{\mapDC} Z$ such that the composition $\mapDC\circ\mapCD$ is continuous. Choose a closed subset $Z'\subseteq Z$. Then $X':=(\mapDC\circ\mapCD)^{-1}(Z')$ is a closed subset of $X$. Observe that $X$ is compact while $Y$ is Hausdorff, thus $\mapCD$ is closed. Therefore $Y':=\mapCD(X')$ is a closed subset of $Y$. It follows from the surjectivity of $\mapCD$ that $Y'=\mapDC^{-1}(Z')$. Thus $\mapDC$ is continuous. 
\end{proof}

\subsection{Positroids and weak separation}\label{sec:weak_sep}
Before we proceed with the final step of the proof, we need to introduce some constructions related to \emph{positroids}; see~\cite{Pos,OPS} for background. Our ultimate goal is to prove \cref{cor:touch_arch,cor:square_face}, which state that under certain hypotheses, we can apply square moves to find either an interior square face or an \emph{$I_j$-arch} (\cref{fig:arch} and \cref{dfn:arch}). Finding such faces bounded by a small number of edges is useful for our proof of the injectivity in \cref{sec:inj} since it allows one to reconstruct cross-ratios of edge weights from the image of $\Measbkn$.

Let $g\in\Bkn$ be a bounded affine permutation. For $q\in\Z$, let
\begin{equation}\label{eq:Grneck_dfn}
  \It_q:=\{g(p)\mid p\in\Z\text{ is such that }p<q\leq g(p)\}.
\end{equation}
We set $\Icalt_g=(\It_q)_{q\in\Z}$. For $q\in[n]$, let $I_q\in{[n]\choose k}$ be obtained from $\It_q$ by reducing all elements modulo $n$. The \emph{Grassmann necklace} of $g$ is the sequence $\Ical_g=(I_1,I_2,\dots,I_n)$. For each $q\in[n]$, consider a total order $\leq_q$ on $[n]$ given by $q\leq_q q+1\leq_q\cdots\leq_q q-1$. For two sets $I=\{i_1<_q i_2 <_q \cdots <_q i_k\}$ and $J=\{j_1<_qj_2<_q\cdots <_q j_k\}$, we write $I\leq_q J$ if $i_r\leq_q j_r$ for all $r\in[k]$. The \emph{positroid} $\Mcal_g$ of $g$ is defined as the collection of all $J\in{[n]\choose k}$ satisfying $I_q\leq_q J$ for each $q\in[n]$.

We say that $I,J\in{[n]\choose k}$ are \emph{weakly separated}~\cite{LeZe} if there do not exist indices $1\leq a<b<c<d\leq n$ such that $a,c\in I\setminus J$ and $b,d\in J\setminus I$ or vice versa.

For $G\in\Gred(g)$ and $j\in[n]$, we let $w_j$ denote the unique neighbor of the degree $1$ boundary vertex $b_j$.
\begin{definition}\label{dfn:arch}
Let $g\in\Bkn$ and $j,\rt\in[n]$. Let $r:=\gb(j-1)\in[n]$. (Here and below the index $j-1$ is taken modulo $n$.) Assume that $\rt\neq j\neq r\neq \rt$. We say that $\rt$ \emph{touches an $I_j$-arch} (with respect to $g$) if there exists a contracted graph $G\in\Gred(g)$ such that the boundary face of $G$ between $b_j$ and $b_{j-1}$ is a pentagon with vertices $(b_j,w_j,b,w_{j-1},b_{j-1})$ for some black interior vertex $b$, and such that the strand labeled $\rt$ passes through the edges connecting $w_j$ to $b$ and $b$ to $w_{j-1}$. See \cref{fig:arch}.
\end{definition}

\begin{figure}
\includegraphics[width=0.2\textwidth]{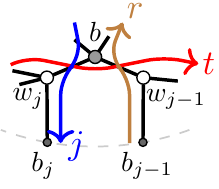} 
  \caption{\label{fig:arch} An $I_j$-arch; see \cref{dfn:arch}.}
\end{figure}

Our notion of an \emph{$I_j$-arch} is closely related to the notion of a \emph{BCFW bridge}; see~\cite{BCFW,abcgpt,LamCDM}. In fact, a bridge is a special case of an arch when either $w_j$ or $w_{j-1}$ has degree $2$; compare \cref{fig:arch} to e.g.~\cite[Figure~7]{crit}. 
We now establish a useful criterion for the existence of an $I_j$-arch.
\begin{lemma}\label{lemma:keystone}
Let $g\in\Bkn$, $j,\rt\in[n]$, $r:=\gb(j-1)$ be such that  $\rt\neq j\neq r\neq \rt$. Then $\rt$ touches an $I_j$-arch if and only if the sets
\begin{equation}\label{eq:JK_keystone}
  J:=I_j\cup\{\rt\}\setminus\{j\} \quad\text{and}\quad R:=I_j\cup\{\rt\}\setminus\{r\}
\end{equation}
belong to $\Mcal_g$ and are weakly separated from all sets in $\Ical_g$.
\end{lemma}
\begin{proof}
We start with the \emph{if} direction. Since $I_j,J,R\in \Mcal_g$, they are all of size $k$, and since $\rt\neq j\neq r\neq \rt$, we have $I_j\neq J\neq R\neq I_j$. In particular, $j,j-1$ are neither loops nor coloops. (Otherwise, either $j$ or $r$ would appear in either none or all of the three sets $I_j,J,R$.) Clearly, $J$ and $R$ are weakly separated from each other. Since they are also weakly separated from all sets in $\Ical_g$, by~\cite[Theorem~1.5]{OPS}, there exists a contracted graph $G\in\Gred(g)$ such that $J,R$ appear as \emph{face labels} of $G$. Here we label the faces of $G\in\Gred(g)$ by $k$-element sets as follows: for each face $F$ of $G$, the label of $F$ contains $s\in[n]$ if and only if $F$ is to the left of the strand terminating at $b_s$. 

Observe that $|J\cup I_j\cup R|=|I_j\cup \{\rt\}|=k+1$. Thus $J,I_j,R$ belong to a \emph{non-trivial black clique} in the sense of~\cite[Section~9]{OPS}. In particular, the faces of $G$ labeled $J,I_j,R$ share a black vertex $b\in V(G)$. 

Since $j,j-1$ are neither loops nor coloops, we have $I_{j-1}\neq I_j\neq I_{j+1}$.  Suppose that $J\neq I_{j+1}$. Then $I_j,J,I_{j+1}$ belong to a \emph{non-trivial white clique}, and thus the corresponding faces of $G$ share a common white vertex, which, since $G$ is contracted, equals $w_j$. If $J=I_{j+1}$ then the two faces labeled by $J$ and $I_j$ still share the degree $2$ vertex $w_j$. Similarly, the faces labeled by $R$ and $I_j$ share $w_{j-1}$. By~\cite[Lemma~9.2]{OPS}, $I_j$ and $J$ share an edge connecting $w_j$ to $b$ while $I_j$ and $R$ share an edge connecting $w_{j-1}$ to $b$. The strand labeled $\rt$ therefore must pass through both of these edges, so $\rt$ touches an $I_j$-arch.

The \emph{only if} direction is a trivial consequence of the results of~\cite{OPS}: if $t$ touches an $I_j$-arch then $J,R$ appear as labels of the faces of $G$ containing $b$, and therefore $J,R$ belong to $\Mcal_g$ and are weakly separated from all sets in $\Ical_g$ by~\cite[Theorem~1.5]{OPS}.
\end{proof}

Next, we apply the above lemma to a particular class of permutations $\gBB$ constructed in \cref{lemma:gBB}.
\begin{definition}\label{dfn:BB_generic}
An $n$-periodic interval partition $\BB$ of $\Z$ is called \emph{generic} if we have 
\begin{equation*}%
  |\B|\leq \min(k-1,n-k) \quad\text{for all $\B\in\BB$.}
\end{equation*}
\end{definition}
\noindent In other words, $\BB$ is generic if and only if $\ovl(\A,\B)=\ovr(\A,\B)=0$ for all $\B\in\BB$ and $\A:=\B-k$. Consequently, $\gBB$ restricts to an order-reversing map $\A\to\B$ for each such pair $(\A,\B)$. For the rest of this subsection, we fix some generic $\BB$. Recall from \cref{notn:cyclic} that for $p,q\in[n]$, $[p,q)$ denotes the corresponding cyclic interval.
\begin{lemma}\label{lemma:Grneck_decsr}
Let $[p,q)\in\BBmod$ and $r\in [p,q)$. Let $j\in[n]$ be equal to $p+q-k-r$ modulo $n$. Then the corresponding element of the Grassmann necklace $\Ical_{\gBB}$ is given by
\begin{equation*}%
  I_j:=[j,p)\sqcup [r,q).
\end{equation*}
Moreover, every element of $\Ical_{\gBB}$ appears in this way for a unique triple $(p,q,r)$.
\end{lemma}
\noindent We note that such Grassmann necklaces have been previously studied in~\cite[Section~4.4]{FG}.
\begin{proof}
Follows from~\eqref{eq:Grneck_dfn} by direct observation.
\end{proof}
\begin{definition}
A set $J\in{[n]\choose k}$ is called \emph{right-aligned} if for each $[p,q)\in\BBmod$, we have
\begin{equation*}%
  J\cap [p,q)=[r,q) \quad\text{for some $r\in[p,q]$.}
\end{equation*}
\end{definition}
\begin{lemma}\label{lemma:right_aligned}
Let $J\in{[n]\choose k}$ be right-aligned. Then $J\in\Mcal_{\gBB}$ and $J$ is weakly separated from all sets in $\Ical_{\gBB}$.
\end{lemma}
\begin{proof}
Let $I_j\in\Ical_{\gBB}$. After a cyclic shift, we may assume $j=1$, thus $I_1=[1,p)\sqcup[r,q)$ for some $[p,q)\in \BBmod$ with $r\in[p,q)$. Our goal is to show that $I_1\leq_1 J$ and that $J$ is weakly separated from $I_1$.  If $J$ does not contain any elements in $[p,r)$ then both claims are clear. Otherwise, let $I_1':=[1,p)$, $J':=J\cap[1,p)$, $I_1'':=[r,q)$, and $J'':=J\cap[p,q)$, thus $J''$ contains an element $s\in [p,r)$. However, since $[p,q)\in\BBmod$ and $J$ is right-aligned, we must have $J''\supseteq I_1''$. On the other hand, $J'\subseteq I_1'$, so $I_1$ and $J$ are weakly separated. Moreover, because $|I_1|=|J|$ and $J$ contains the last $q-r$ elements of $I_1$, we get $I_1\leq_1 J$.
\end{proof}

\begin{corollary}\label{cor:touch_arch}
Let $j\in[n]$ and $[s,s')\in\BBmod$ be such that $I_j\cap [s,s')=\emptyset$. Then $\rt:=s'-1$ touches an $I_j$-arch with respect to $\gBB$.
\end{corollary}
\begin{proof}
Let $I_j=[j,p)\sqcup[r,q)$ with $r\in[p,q)\in\BBmod$ be as in \cref{lemma:Grneck_decsr}. Observe that $r=\gbBB(j-1)$. The sets $J,R$ given by~\eqref{eq:JK_keystone} are clearly right-aligned. By \cref{lemma:right_aligned}, they satisfy the conditions of \cref{lemma:keystone}.
\end{proof}

\begin{corollary}\label{cor:square_face}
Assume that $k\leq n-2$. Let $[p_1,q_1),[p_2,q_2),[p_3,q_3),[p_4,q_4)\in\BBmod$ be four disjoint intervals, listed in clockwise order. Then there exists a contracted graph $G\in\Gred(\gBB)$ containing a square face whose edges are labeled by $\{\rt_1,\rt_2\}$, $\{\rt_2,\rt_3\}$, $\{\rt_3,\rt_4\}$, $\{\rt_1,\rt_4\}$ with $\rt_j\in[p_j,q_j)$ for each $j=1,2,3,4$.
\end{corollary}
\begin{proof}
Consider all right-aligned subsets of $[n]$ whose intersection with $[p_j,q_j)$ is nonempty for each $j=1,2,3,4$. Clearly, such subsets can have any size between $4$ and $n$. Let $I$ be such a set of size $k+2$, and for $j=1,2,3,4$, let $I\cap [p_j,q_j)=[\rt_j,q_j)$, where $\rt_j\in[p_j,q_j)$. The sets $I\setminus \{\rt_i,\rt_j\}$ for $1\leq i<j\leq 4$ are all right-aligned. Thus they belong to $\Mcal_{\gBB}$ and are weakly separated from all elements of $\Ical_{\gBB}$ by \cref{lemma:right_aligned}. The result follows by combining~\cite[Proposition~3.2]{OPS} with~\cite[Theorem~1.3]{OPS}.
\end{proof}

\subsection{Injectivity}\label{sec:inj}
Our final goal is to show that the map $\mapDC:\Delta_{2,n}\to\Ctnn_{k,n}$ constructed in \cref{thm:Measf:Cyc_to_Ctnn} is injective, which is closely related to the \emph{injectivity conjecture} for critical cells; see~\cite[Conjecture~4.3]{crit}. It was proved for $\Ctp_{k,n}$ in~\cite[Theorem~4.4]{crit}. What we need is an extension of that result to the closure $\Ctnn_{k,n}$ of $\Ctp_{k,n}$ which turns out to be more subtle.
\begin{theorem}
  The map $\mapDC:\Delta_{2,n}\to\Ctnn_{k,n}$ is a homeomorphism.
\end{theorem}
\begin{proof}
Since $\Delta_{2,n}$ is compact, $\Ctnn_{k,n}$ is Hausdorff, and $\mapDC$ is a continuous surjection, it remains to show that $\mapDC$ is an injection. Thus for $\bth\in\Cyc(\Pakn)$, our goal is to show that the point $\mapCD(\bth)\in\Delta_{2,n}$ can be uniquely reconstructed from $\Measbkn(\bth)\in\Ctnn_{k,n}$. Let $\BB_\bth$, $\gbth$, $G''\in\Gred(\gbth)$, and $\wt''_\bth$ be as in \cref{sec:DC} and \cref{prop:graph_reduction}. 

First, observe that $\BB_\bth$ need not be generic in the sense of \cref{dfn:BB_generic} since $\gbth$ may have some coloops and some indices $j\in\Z$ satisfying $\gbth(j)=j+1$. The corresponding strands form isolated connected components of the reduced strand diagram of $\gbth$. We remove these components using the \emph{factorization} procedure from~\cite[Section~4.4]{crit}. Thus the problem reduces to the case where $\BB_\bth$ is generic, which allows us to apply the results of \cref{sec:weak_sep}.

Let $\BBmod=(\Bmod_1,\Bmod_2,\dots,\Bmod_m)$. The point $\mapCD(\bth)$ records the side length ratios of a (possibly degenerate) inscribed $m$-gon $\Rbth$. For $p,q\in[m]$, let $\dist(p,q)$ denote the distance between the corresponding vertices of $\Rbth$. The ratio of any two such distances can be computed from $\mapCD(\bth)$; see \cref{rmk:polygons_k_n}. Recall from~\eqref{eq:weights} that the edge weights of the graph $G''$ are proportional to the distances between the vertices of $\Rbth$. More precisely, if an interior (that is, not incident to a boundary vertex) edge $e$ of $G''$ is labeled by $\{s,t\}$ then $s\in\Bmod_p$ and $t\in\Bmod_q$ belong to different cyclic intervals in $\BBmod_\bth$, and the weight $\wt''_\bth(e)$ is proportional (compared to the other edges sharing a black vertex with $e$) to $\dist(p,q)$. 

As explained in \cite[Section~9]{crit}, for any face $F$ of $G''$, the alternating ratio of the edge weights that appear on the boundary of $F$ may be reconstructed from $\Measbkn(\bth)$ using the \emph{left twist} of Muller--Speyer; see~\cite[Corollary~5.11]{MuSp}. We will be interested in two kinds of faces of $G''$: $I_j$-arches as in \cref{dfn:arch} and interior square faces as in \cref{cor:square_face}.

Let $s,t\in[m]$ and $j\in\Bmod_s$ be such that $I_j\cap \Bmod_t=\emptyset$. By \cref{cor:touch_arch}, some element in $\Bmod_t$ touches an $I_j$-arch. Then for $r\in[m]$ such that $\gbbth(j-1)\in\Bmod_r$, we find that the ratio
\begin{equation}\label{eq:Bmod_arch}
  \frac{\dist(s,t)}{\dist(r,t)}
\end{equation}
may be recovered from $\Measbkn(\bth)$. %

Similarly, assume $k\leq n-2$ and let $\Bmod_p,\Bmod_q,\Bmod_t,\Bmod_s\in\BBmod$ be four disjoint intervals listed in clockwise order. Then by \cref{cor:square_face}, the \emph{cross-ratio}
\begin{equation}\label{eq:Bmod_square}
  \frac{\dist(p,q)\cdot \dist(t,s)}{\dist(q,t)\cdot \dist(p,s)}
\end{equation}
may be recovered from $\Measbkn(\bth)$. 
 In fact, since the four corresponding vertices of $\Rbth$ lie on a circle or on a line, the ratios
\begin{equation}\label{eq:Bmod_square3}
  (\dist(p,q)\cdot \dist(t,s)) :   (\dist(p,t)\cdot \dist(q,s)) :   (\dist(p,s)\cdot \dist(q,t))
\end{equation}
can all be recovered from $\Measbkn(\bth)$ using standard relations for cross-ratios.

Recall that we have $\BBmod=(\Bmod_1,\Bmod_2,\dots,\Bmod_m)$. Consider a directed graph $D$ on $[m]$ with edges $s\to r$ whenever there exists $j\in\Bmod_s$ such that $\gbbth(j-1)\in\Bmod_r$. Thus the ratio in~\eqref{eq:Bmod_arch} may be recovered from $\Measbkn(\bth)$ for all $t\in[r+1,s-1]$. Clearly, each vertex of $D$ has at least one outgoing arrow. Moreover, since $\BBmod$ is generic, we see that each vertex $s$ of $D$ has an outgoing arrow $s\to r$ for $r\neq s,s-1$ (modulo $m$). Finally, by comparing $\gbbth(j)$ to $\gbbth(j+1)$, we see that if $D$ has an arrow $s\to r$ then $D$ also has at least one of the following arrows: $s\to r+1$, $s+1\to r$, $s+1\to r+1$.

By \cref{lemma:polygon}, it suffices to recover the ratio $\dist(s,{s-1}):\dist({s-1},{s-2})$ from $\Measbkn(\bth)$ for each $s\in[m]$. This task is trivial when $k=n-1$, thus let us assume that $k\leq n-2$.  As shown above, there exists $r\neq s,s-1$ such that $D$ contains an arrow $s\to r$. If $r=s-2$ then we are done, thus assume $r\neq s,s-1,s-2$ and let $t\in [r+2,s-1]$. We know that $D$ contains another arrow $s'\to r'$ for $s'\in\{s,s+1\}$, $r'\in\{r,r+1\}$.   
 From~\eqref{eq:Bmod_arch}, we recover the ratios
\begin{equation}\label{eq:rat1}
  \dist(s,t):\dist(r,t),\quad \dist(s,r'):\dist(r,r'),\quad \dist(s',t):\dist(r',t),\quad \dist(s',s):\dist(r',s),
\end{equation}
some of which may coincide or be equal to $1$ if $s=s'$ or $r=r'$.

Suppose first that $s'=s+1$ and $r'=r+1$. Using~\eqref{eq:Bmod_square3}, we recover the ratios
\begin{equation}\label{eq:rat_square1}
  \begin{aligned}
  &(\dist(s,s+1)\cdot \dist(r,r+1)) :   (\dist(s,r)\cdot \dist(s+1,r+1)) :   (\dist(s,r+1)\cdot \dist(s+1,r)),\\
  &(\dist(s+1,r)\cdot \dist(r+1,t)) :   (\dist(s+1,r+1)\cdot \dist(r,t)) :   (\dist(s+1,t)\cdot \dist(r,r+1)).
\end{aligned}
\end{equation}
Combining~\eqref{eq:rat1} with~\eqref{eq:rat_square1}, we recover 
\begin{equation*}%
  \dist(s,s+1):\dist(s+1,r):\dist(r,r+1):\dist(s,r+1).
\end{equation*}
By \cref{lemma:polygon}, we recover the (possibly degenerate) inscribed quadrilateral with vertices $s,s+1,r,r+1$. 
The cases $s'=s$, $r'=r+1$ and $s=s+1$, $r'=r$ are handled similarly. In the former case, we recover the inscribed triangle with vertices $s,r,r+1$, and in the latter case, we recover the inscribed triangle with vertices $s,s+1,r$. (When we say ``we recover a polygon'' we mean that the ratio of any two of its side lengths may be recovered from $\Measbkn(\bth)$.) Thus we have recovered a possibly degenerate inscribed polygon $R$ whose vertex set $\Vert(R)$ contains $s$ and $r$. By~\eqref{eq:Bmod_arch}, for each $t'\in[r+1,s-1]$, we recover the ratio $\dist(s,t):\dist(r,t)$, and thus the possibly degenerate inscribed polygon with vertex set $\Vert(R)\cup [r,s]$ is recovered. In particular, the ratio $\dist(s,{s-1}):\dist({s-1},{s-2})$ is recovered.
\end{proof}
 
\bibliographystyle{alpha} 
\bibliography{crit_polyt}

\end{document}